\UseRawInputEncoding
\documentclass[11pt]{amsart}
\usepackage[UKenglish]{isodate}
\usepackage{amsmath}
\usepackage{amsthm}
\usepackage{amsbsy}
\usepackage{amssymb}
\usepackage{amsfonts}
\usepackage[dvips]{lscape}
\usepackage{xcolor}
\usepackage{amscd}
\usepackage[all,cmtip]{xy}
\usepackage{euscript}
\usepackage{caption}  
\usepackage{parskip}
\usepackage{enumerate}
\usepackage{yfonts}
\usepackage{xcolor}
\usepackage{graphicx}
\usepackage{fancyhdr}
\usepackage{tikz-cd}
\usepackage{colortbl}

\usepackage{fullpage}
\usepackage{verbatim}

\usetikzlibrary{positioning}
\usetikzlibrary{decorations.pathreplacing}

\usepackage[margin=1in]{geometry}
\usepackage[expansion=false]{microtype}

\theoremstyle{plain}

\theoremstyle{plain}
\newtheorem{theorem}{Theorem}[section]

\newtheorem{proposition}[theorem]{Proposition}
\newtheorem{lemma}[theorem]{Lemma}
\newtheorem{corollary}[theorem]{Corollary}

\newtheorem{conjecture}[theorem]{Conjecture}

\theoremstyle{remark}
\newtheorem{remark}[equation]{Remark}

\theoremstyle{definition}

\renewcommand{\Im}{\textup{Im}}

\newcommand{\HH}{\mathbb{H}}

\renewcommand{\vec}[1]{\boldsymbol{#1}}

\newcommand{\R}{\mathbf{R}}
\newcommand{\Z}{\mathbf{Z}}
\newcommand{\A}{\mathbf{A}}
\newcommand{\CC}{\mathbf{C}}
\newcommand{\X}{\mathbf{X}}
\newcommand{\PP}{\mathcal{P}}
\newif\iffinalrun
\iffinalrun
\else
 \fi

\iffinalrun
  \newcommand{\need}[1]{}
  \newcommand{\mar}[1]{}
\else
  \newcommand{\need}[1]{{\tiny *** #1}}
  \newcommand{\mar}[1]{\marginpar{\raggedright\tiny Fix Me:  #1 }}\fi

\usepackage{amssymb,amsmath,amsfonts,amsthm,epsfig,amscd,stmaryrd}
\usepackage{stmaryrd}

\usetikzlibrary{arrows}

\pgfdeclarelayer{background}
\pgfsetlayers{background,main}

\pgfmathsetmacro{\myxlow}{-2}
\pgfmathsetmacro{\myxhigh}{2}
\pgfmathsetmacro{\myiterations}{6}

\title{Higher Fourier interpolation on the plane }
\author{Naser Talebizadeh Sardari}
\address{Penn State department of Mathematics, McAllister Building, Pollock Rd, State College, PA 16802 USA}
\email{nzt5208@psu.edu}
\thanks{  N.T.S. thanks  Professors Stephen D. Miller, Carlo Pagano, Danylo Radchenko, Zeev Rudnick, Peter Sarnak and Masoud Zargar for valuable comments. N.T.S. was supported partially by the National Science Foundation under Grant No. DMS-2015305,
and is grateful to Max Planck Institute for Mathematics in Bonn for its hospitality and financial support.}

\begin{document}

\begin{abstract} Let  $l\geq 6$  be any integer, where $l\equiv 2$ mod $4$.  Suppose that $\mu(\tau)d\tau$ is a measure with bounded  variation and is supported on a compact subset of the complex plane,  where  $\Im(\tau),\Im(-\frac{1}{\tau})>\sin\left(\frac{\pi}{l}\right).$  Let $f(x)=\int e^{i\pi \tau |x|^2}d\mu(\tau)$ and $\mathcal{F}(f)$ be its Fourier transform,  where  $x\in \R^2.$   For every integer  $k\geq 0$ and $x\in \R^2,$
we  express  $f(x)$  in terms of the values of $\frac{d^k f}{du^k}$ and $\frac{d^k \mathcal{F}(f)}{du^k}$
 at $u=\frac{2n}{\lambda},$ where $n$ is a non-negative integer,  $u=|x|^2$ and $\lambda=2\cos\left(\frac{\pi}{l}\right).$   
 
 We show that the condition $\Im(\tau),\Im(-\frac{1}{\tau})>\sin\left(\frac{\pi}{l}\right)$ is optimal.
We also identify the summation formulas among the values of $\frac{d^k f}{du^k}$ and $\frac{d^k \mathcal{F}(f)}{du^k}$
 at $u=\frac{2n}{\lambda},$ with the space of holomorphic modular forms of weight $2k+1$ of the Hecke triangle group $(2,l,\infty)$.  
Using our  formulas for $l=6$ and developing new methods, we prove a conjecture of  Cohn,  Kumar,  Miller,  Radchenko and   Viazovska~\cite[Conjecture 7.5]{Maryna3}. This conjecture  was motivated by the universal optimality of the hexagonal lattice.
\end{abstract}
\maketitle
\section{Introduction}
Let $f:\R^d\to \mathbf{C}$ be a radial Schwartz function. 
 Let $\mathcal{F}(f)(\xi)$ be the  Fourier transformation of $f$
\[
\mathcal{F}(f)(\xi):=\int_{\R^d}f(x)e^{2\pi i \left<x,\xi\right>} dx.
\]
Radchenko and Viazovska recently proved an elegant formula~\cite{inter} for $d=1$ that expresses the value of $f$ at any
given point in terms of the values of $f$ and $\mathcal{F}(f)$ on the
set $\{ \sqrt{|n|}:n\in \Z\}.$ Their method can be generalized to every $d\geq1.$
\\

 Cohn,  Kumar,  Miller,  Radchenko and  Viazovska~\cite[Theorem 1.7]{Maryna3} developed  new Fourier interpolation formulas  to prove the optimality of the $E_8$ and the leech lattice. Their formulas express the value of $f$ at any
given point in terms of the values of $f,$  $\mathcal{F}(f),$ $\frac{d f}{du}$ and $\frac{d \mathcal{F}(f)}{du}$   on the
set $\{ \sqrt{2|n|}:n\geq n_d, \text{ and }n\in \Z\},$ where $u=|x|^2,$ and $(d,n_d)=(8,1), (24,2).$ 
\\

In the last section of their paper the authors ask two deep questions. The first question speculate~\cite[Open problem 7.1]{Maryna3} the existence of interpolation formulas using the values of the higher derivatives $\frac{d^k f}{du^k}$ and $\frac{d^k \mathcal{F}(f)}{du^k}.$ They state that that their methods cannot apply to $k \geq 2$ without serious modification.
In the second question, they speculate~\cite[Open problem 7.3]{Maryna3} the existence of Fourier interpolation formulas for other discrete sets. 
They state that  the special nature of the interpolation points $\sqrt{2n}$ plays an essential role in their proofs. 
\\

One particular case of discrete sets is related to the optimality of the hexagonal lattice.   
They conjectured the following based on their numerical experiments. We discuss  its relation to the hexagonal lattice  in section~\ref{opin}.

\begin{conjecture}\cite[Conjecture 7.5]{Maryna3} Let $r_1, r_2,\dots $ be the positive real numbers of the form
$
(4/3)^{1/4}\sqrt{
j^2+jk+k^2},$ where j and k are integers. Then radial Schwartz
functions $f :\R^2\to \R$ are not uniquely determined by the values of $f(r_n),$  $\mathcal{F}(f)(r_n),$ $\frac{d f}{du}(r_n)$ and $\frac{d \mathcal{F}(f)}{du}(r_n)$ for $n\geq1.$ 
\end{conjecture}

Our main goal in the paper is to address these two questions. We develop new interpolation formulas using the values of the higher derivatives on new discrete sets. In particular,  we prove the above conjecture in Theorem~\ref{mainconj}. We restrict our formulas to $d=2$ for clarity of exposition.


\subsection{Interpolation with values} Suppose that $l\geq 6$ is an integer and  $l\equiv 2$ mod $4$. In this section, we discuss our main theorem in the special case $k=0$ and $f(x)=e^{i\pi \tau |x|^2},$ where $\Im(\tau),\Im(-\frac{1}{\tau})>\sin\left(\frac{\pi}{l}\right).$
 Let $\mathcal{F}(f)(\xi)$ be the  Fourier transformation of $f$
\[
\mathcal{F}(f)(\xi):=\int_{\R^2}f(x)e^{2\pi i \left<x,\xi\right>} dx.
\]
It is well-known that $\mathcal{F}(f)(\xi)=\frac{i}{\tau}e^{i\pi \frac{-1}{\tau} |\xi|^2}$. We define
\[
f^{\varepsilon}:=f+\varepsilon\mathcal{F}(f),
\]
where $\varepsilon=\pm1.$ Note that $f^{\varepsilon}$ is an eigenfunction of the Fourier transformation with eigenvalue $\varepsilon$ and 
\[
f=\frac{f^++f^-}{2}.
\]

 Next, we introduce a family of  $\pm1$ eigenfunctions for the Fourier transformation.  Let $\Gamma$ be the  triangle group $(2,l,\infty);$ see Figure~\ref{heckef}. Let 
 \begin{equation}\label{dim1}
 d_{\varepsilon}= \begin{cases} 0,  \text{ if } \varepsilon=+1, \\ 1, \text{ if } \varepsilon=-1.\end{cases}
 \end{equation}
Later, we identify $d_{\varepsilon}$ with the dimension of a specific space of modular forms of $\Gamma$.  
 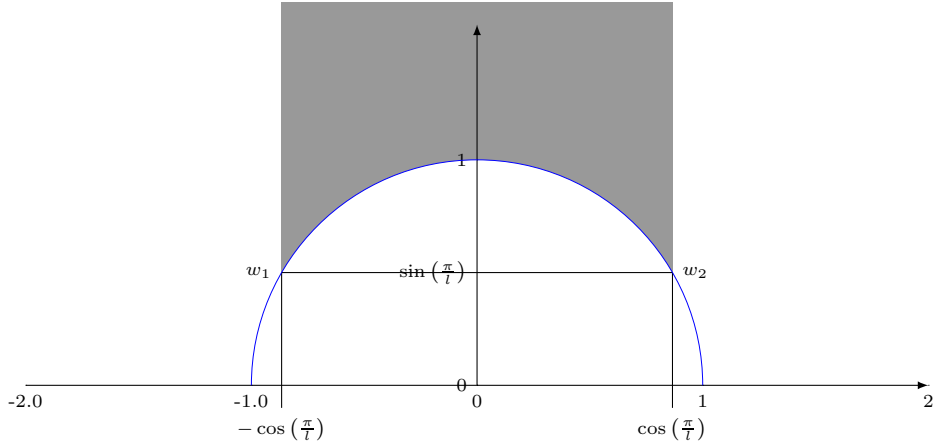
\begin{figure}
\centering
\begin{tikzpicture}[scale=3]
    \draw[-latex](\myxlow,0) -- (\myxhigh ,0);
    \pgfmathsetmacro{\succofmyxlow}{\myxlow+1}
    \foreach \x in {\myxlow,\succofmyxlow,...,\myxhigh}
    {   \draw (\x,-0) -- (\x,-0.0) node[below,font=\tiny] {\x};
    }
        {   \draw (-0.866,-0.1)node[below,font=\tiny]{$-\cos\left(\frac{\pi}{l}\right)$}--(-0.866,0.5)node[left,font=\tiny]{$w_1$}--(0,0.5)node[left,font=\tiny]{$\sin\left(\frac{\pi}{l}\right)$}  --(0.866,0.5)node[right,font=\tiny]{$w_2$} -- (0.866,-0.1) node[below,font=\tiny] {$\cos\left(\frac{\pi}{l}\right)$};
    }
    
    \foreach \y  in {0,1}
    {   \draw (0,\y) -- (-0.0,\y) node[left,font=\tiny] {\pgfmathprintnumber{\y}};
    }
    \draw[-latex](0,-0.0) -- (0,1.6);
    \begin{scope}   
        \clip (\myxlow,0) rectangle (\myxhigh,1.2);

            {   \draw[very thin, blue] (1,0) arc(0:180:1);
            }   
    \end{scope}
    \begin{scope}
            \begin{pgfonlayer}{background}
            \clip (-0.866,0) rectangle (0.866,1.7);
            \clip   (1,1.7) -| (-1,0) arc (180:0:1) -- cycle;
            \fill[gray,opacity=0.8] (-1,-1) rectangle (1,2);
        \end{pgfonlayer}
    \end{scope}
\end{tikzpicture}
\captionof{figure}{Fundamental domain for the Hecke triangle $(2,l,\infty)$}\label{heckef}
\end{figure}
Let $\phi_{n,0}^{\varepsilon}(z)$ the weakly holomorphic modular form of weight $1,$ multiplier $\varepsilon$ and of depth $n\geq d_{\varepsilon}$ defined in~\eqref{defeqqq}. It follows that we only have weakly holomorphic modular forms of weight 1 if $l\equiv 2$ mod 4 and $l\geq 3$.  In particular when $l=6$,
\[
\phi_{0,0}^{+}(z):=\sum_{x,y\in \Z} e^{2\pi i z \frac{( x^2+xy+y^2)}{\sqrt{3}}}
\]
 is the theta series associated to the normalized hexagonal lattice in $\R^2.$ 
\\

 Let \[w_1:=-\cos\left(\frac{\pi}{l}\right)+i\sin\left(\frac{\pi}{l}\right), \text{ and } w_2=\cos\left(\frac{\pi}{l}\right)+i\sin\left(\frac{\pi}{l}\right).\] We define
\begin{equation}\label{1bas}
a_n^{\varepsilon}(x):= \frac{1}{\lambda}\int_{w_1}^{w_2}\phi_{n,0}^{\varepsilon}(z)e^{\pi i z |x|^2} dz,
\end{equation}
where $\lambda=2\cos\left(\frac{\pi}{l}\right).$   We show that $a_n^{\varepsilon}(x)$ is a radial Schwartz $\varepsilon$ eigenfunctions of the Fourier transformation in $\R^2$.  
 Moreover, for $m,n \geq d_{\varepsilon}$
\[
a_n^{\varepsilon}\left(\sqrt{\frac{2m}{\lambda}}\right)=\delta_{m,n}:=\begin{cases} 1, \text{ if }m=n, \\ 0, \text{ otherwise.}\end{cases}
\]  We state a version of our interpolation formula for $f^{\varepsilon}.$ 

\begin{corollary}\label{mainthm}  
We have
\begin{equation}\label{mainform}f^{\varepsilon}(x)=\sum_{n\geq d_{\varepsilon}}a_n^{\varepsilon}(x)f^{\varepsilon}\left(\sqrt{\frac{2n}{\lambda}}\right).
\end{equation}
\end{corollary}

Corollary~\ref{mainthm} is a special case of Theorem~\ref{mainthmder}. 
Corollary~\ref{mainthm} generalizes the interpolation formula of Radchenko and  Viazovska~\cite{inter} from $l=\infty$ and  $d=1$ to every $l$ and  $d=2,$ where $l\geq 6$ and  $l\equiv 2 $ mod $4.$ \footnote{We have learned from Radchenko that he and Viazovska 
have also considered other triangle groups $(2,l,\infty)$
in this context.}
They used the values at $\sqrt{n}$ and the weakly holomorphic modular of weight $3/2$ which are invariant by the theta group (Hecke triangle group $(2,\infty,\infty)$). 
We use the weakly holomorphic modular forms of weight $1$ which are invariant by the Hecke triangle group $(2,l,\infty)$. Our integration is  over the arc between $w_1$ and $w_2$ and it is   the whole semicircle when $l=\infty$ as defined in~\cite{inter}. Next, we discuss some new features that occur in our work which are  different from that of Radchenko and  Viazovska~\cite{inter}.
\\

 Radchenko and  Viazovska proved their interpolation formula for every  $\tau$ in the upper half-plane and from this they deduced that formula~\eqref{mainform} holds for every  
 radial Schwartz function $f$ in $\R$. We show that formula~\eqref{mainform} is false for some finite  linear combination of Gaussian (unless $l=\infty$). In fact, the bound  $\Im(\tau),\Im(-\frac{1}{\tau})>\sin\left(\frac{\pi}{l}\right)$  is optimal. We  show this by constructing for every $\epsilon>0$ a finite linear combination of Gaussian $\sum a_ie^{i\pi \tau_i r^2}$ which vanishes at all points $\sqrt{\frac{2n}{\lambda}},$ where  $\Im(\tau_i),\Im(-\frac{1}{\tau_i})>\sin\left(\frac{\pi}{l}\right)-\epsilon$ for every $i$.  We realize these functions conceptually by monodromy around $w_1$ and its orbit  $\Gamma w_1$ in section~\ref{monodromy}. We  introduce them explicitly and state our next theorem. 
 Let 
\[
S:=\begin{bmatrix}  0 & 1 \\-1 &0
\end{bmatrix}  , \text{ and }T:=\begin{bmatrix}  1 & \lambda \\0 &1
\end{bmatrix} \text{ and } V:=TS. 
\]
It is well-known that $\Gamma$ is generated by $T$ and $S$, and $V^l=I.$ Let $|^{\varepsilon}_k \gamma$ denote the slash operator of weight $k$ and multiplier $\varepsilon$ associated to $\gamma\in \Gamma$. In particular, 
\[
\begin{split}
f(z)|^{\varepsilon}_k S=\varepsilon \left(\frac{ i}{z}\right)^kf\left(\frac{-1}{z}\right) ,
\\
f(z)|^{\varepsilon}_k T=f\left(z+ \lambda\right).
\end{split}
\] 
We define
\[
r^{\varepsilon}(\gamma,\tau;x):=e^{i\pi \tau |x|^2}|^{-\varepsilon}_1 (T^{-1}-I)(1+V+\dots+V^{l-1})\gamma,
\]
where $\gamma\in \Gamma,$  $\tau$ is in the upper half-plane and $x\in \R^2$. 
 \begin{corollary}\label{mainthm22} $r^{\varepsilon}(\gamma,\tau;x)$ is an eigenfunction of the Fourier transformation with respect to $x$ with eigenvalue $\varepsilon$  for any $\gamma\in \Gamma.$ Moreover, 
\[
r^{\varepsilon}\left(\gamma,\tau;\sqrt{\frac{2n}{\lambda}}\right)=0
\]
for every  integer $n\geq0.$
\end{corollary}
Corollary~\ref{mainthm22} is a special case of Theorem~\ref{mainthm2}.
Note that if $\tau$ is near $w_2$ and $\gamma=id$ then 
\[
r^{\varepsilon}(id,\tau;x)=\sum_{i} \alpha_i e^{i\pi \tau_i |x|^2},
\]
where
$\Im(\tau_i),\Im(-\frac{1}{\tau_i})>\sin\left(\frac{\pi}{l}\right)-\epsilon.$ This shows that the bound  $\Im(\tau),\Im(-\frac{1}{\tau})>\sin\left(\frac{\pi}{l}\right)$  is optimal. 
We use the following key identity in the group algebra $\Z[PSL_2(\R)]$
\[
S(T^{-1}-I)(1+V+\dots+V^{l-1})=-(T^{-1}-I)(1+V+\dots+V^{l-1}).
\]
for the proof of Theorem~\ref{mainthm2}.
\begin{remark}
Note that the number of integers $n$ such that $\sqrt{\frac{2n}{\lambda}}<X$  is about $\cos\left(\frac{\pi}{l}\right)X^2.$ Radchenko and  Viazovska~\cite{inter} in their interpolation formula for   $\R$ uses the values of  $f^{\varepsilon}$ at $\sqrt{n}$ which uses $X^2$ number of points less than $X$.  After stating~\cite[Conjecture 7.5]{Maryna3}, the authors speculate that any interpolation formula in $\R^2$ for all Schwartz functions should also contain at least $X^2$ nodes less than $X$ based on the interpolation formulas in dimensions 1,8 and 24. Corollary ~\ref{mainthm} and Corollary~\ref{mainthm2} are compatible with this speculation. As we introduce a refined class of functions where the interpolation formula holds with given values at  $\sqrt{\frac{2n}{\lambda}}$.\end{remark}

\subsection{Interpolation with higher   derivatives}
 Let $u:=|x|^2$ and $k\geq 0.$ Suppose that  
 \[f(x)=\int e^{i\pi \tau |x|^2}d\mu(\tau),\]  where $\mu$ is a measure with bounded  variation  and  supported on a compact subset of $\Im(\tau),\Im(-\frac{1}{\tau})>\sin\left(\frac{\pi}{l}\right)$ . We may consider $f$ and $\mathcal{F}(f)$ as a smooth function of $u$ in $\R^{+}.$
Next, we develop an interpolation formula for $f$ using  the values of the  $k$-derivatives $\frac{d^k}{du^k} $ of $f$ and $\mathcal{F}(f)$ at $u=\frac{2n}{\lambda}>0$. 
\\

Let $M_k^{\varepsilon}(\Gamma)$ be the space of weight $k$ modular forms with multiplier $\varepsilon.$
Let $d(\varepsilon,k)=\dim \left(M_{2k+1}^{-\varepsilon}(\Gamma)\right).$ In section~\ref{basisk},  we define the space of weakly holomorphic modular forms  by allowing a pole at cusp $\infty.$
 We  introduce a unique weakly holomorphic modular form of weight $-2k+1$  of $\Gamma$  for every $n\geq d(\varepsilon,k)$ and denote it by $\phi^{\varepsilon}_{n,k}(z).$ We define
\[
a_{n,k}^{\varepsilon}(x):=  \frac{1}{\lambda}\int_{w_1}^{w_2}\phi_{n,k}^{\varepsilon}(z)\frac{e^{\pi i z |x|^2}}{(iz\pi)^k}  dz.
\]
We show that $a_{n,k}^{\varepsilon}(x)$ is a radial Schwartz $\varepsilon$ eigenfunctions of the Fourier transformation in $\R^2$.  Moreover, for $m,n\geq d(\varepsilon,k)$
\begin{equation}\label{baseff}
\frac{d^k}{du^k} a_{n,k}^{\varepsilon}\left(\sqrt{\frac{2m}{\lambda}}\right)
=\delta(m,n).
\end{equation}
  In particular when $k=0$,   $a_{n,0}^{\varepsilon}(x)= a_{n}^{\varepsilon}(x),$ which was defined in \eqref{1bas}, and  $d(\varepsilon,0)=d_\varepsilon$ which was defined in \eqref{dim1}.  We state a version of our interpolation formula for $f$ with the higher derivatives. 
\begin{theorem}\label{mainthmder} Let $k\geq 0$ and $f^{\varepsilon}$ and $a^{\varepsilon}_{n,k}$ be as above.
We have 
\[f^{\varepsilon}(x)=\sum_{n\geq d(\varepsilon,k)}a^{\varepsilon}_{n,k}(x) \frac{d^k}{du^k}  f^{\varepsilon}\left(\sqrt{\frac{2n}{\lambda}}\right).
\] 
\end{theorem}
We prove Theorem~\ref{mainthmder} in section~\ref{proofm}.
\\

Next, we show that the condition  $\Im(\tau),\Im(-\frac{1}{\tau_i})>\sin\left(\frac{\pi}{l}\right)$ is optimal. 
We define
\[
r^{\varepsilon}_k(\gamma,\tau;x):=\frac{e^{i\pi \tau |x|^2}}{(i\pi\tau)^k}|^{-\varepsilon}_{2k+1} (T^{-1}-I)(1+V+\dots+V^{l-1})\gamma,
\]
where $\gamma\in \Gamma,$  $\tau$ is in the upper half-plane and $x\in \R^2$. 

\begin{theorem}\label{mainthm2} $r^{\varepsilon}_k(\gamma,\tau;x)$ is an eigenfunction of the Fourier transformation with respect to $x$ with eigenvalue $\varepsilon$ for any $\gamma\in \Gamma.$ Moreover, 
\[
\frac{d^k}{du^k}r_k^{\varepsilon}\left(\gamma,\tau;\sqrt{\frac{2n}{\lambda}}\right)=0.
\]
for every  integer $n\geq0.$
\end{theorem}
We prove Theorem~\ref{mainthm2} in section~\ref{monodromy}
\\

Theorem~\ref{mainthmder} generalizes Corollary~\ref{mainthm} to higher derivatives and it implies Corollary~\ref{mainthm}  when $k=0.$ We also note that $d(\varepsilon,k)>0$  is growing linearly for $k\geq 1.$ This means that there are relations among $\frac{d^k}{du^k}  f^{\varepsilon}\left(\sqrt{\frac{2n}{\lambda}}\right).$ We describe the space of relations completely in section~\ref{obs}.

\subsection{Obstructions for Fourier interpolation}\label{obs} Our theorem in this section holds for every radial Schwartz function $f(x)$ on $\R^2$. We introduce a complete family of linear obstructions to the Fourier interpolation with $k$-th derivatives $\frac{d^k}{du^k}  f^{\varepsilon}\left(\sqrt{\frac{2n}{\lambda}}\right).$  We show that the obstructions are associated to the modular forms of weight $2k+1$ and multiplier $-\varepsilon.$ 
\\

\begin{theorem}\label{obsthm}  Suppose that $g(z)=\sum_{n\geq 0} b_n q^n \in M_{2k+1}^{-\varepsilon}(\Gamma),$ where $q=e^{\frac{2\pi i z}{\lambda}}$ and $f(x)$ is a radial Schwartz function. We have 
\[
\sum_{n\geq 0}b_n \frac{d^k}{du^k} f^{\varepsilon}\left(\sqrt{\frac{2n}{\lambda}}\right)=0.
\]
\end{theorem}
\begin{proof}
Since $g(z)$ is a holomorphic modular forms. By Hardy's upper bound, its coefficients satisfy the  polynomial bound $|b_n|\ll n^{2k+\varepsilon}.$ Since $f$ is a Schwartz function then the following series is absolutely convergent
\[
\sum_{n\geq 0}b_n \frac{d^k}{du^k} f^{\varepsilon}\left(\sqrt{\frac{2n}{\lambda}}\right).\]
We want to show that the above tempered distribution is zero on the space of radial Schwartz function. It is enough to prove it for a dense subset of radial Schwartz functions. Following \cite[section 6]{inter}, it is enough to prove it for the normalized complex Gaussian $G_k(z,x):=\frac{e^{i\pi z|x|^2}}{(iz\pi)^k}$ where $z\in \HH.$ Note that

\begin{equation}\label{mfe}g(z)=-\varepsilon\left(\frac{i}{z}\right)^{2k+1}g\left(\frac{-1}{z}\right).\end{equation}

It is easy to check that
 \[\frac{d^k}{du^k}G_k(z,x)=G_0(z,x),\] 
and,
\[
\frac{d^k}{du^k} \mathcal{F}G_k\left(z,x\right)=\left(\frac{i}{z}\right)^{2k+1}G_0\left(\frac{-1}{z},x\right),
\]
where $u=|x|^2$.
 We rewrite~\eqref{mfe} as
\[
\sum_{n\geq 0} b_n \frac{d^k}{du^k} \left( G_k\left(z,\sqrt{\frac{2n}{\lambda}} \right)+\varepsilon \mathcal{F}G_k\left(z,\sqrt{\frac{2n}{\lambda}}\right)  \right)=0.
\]
This implies that 
\[
\sum_{n\geq 0}b_n \frac{d^k}{du^k} f^{\varepsilon}\left(\sqrt{\frac{2n}{\lambda}}\right)=0
\]
 for  \(f(x)=G_k(z,x)^{\varepsilon}\).  This completes the proof of our theorem.
\end{proof}
\begin{remark}
We note Theorem~\ref{obsthm} imposes $d(\varepsilon,k)=\dim \left(M_{2k+1}^{-\varepsilon}(\Gamma)\right)$ independent linear equations on the values of $ \frac{d^k}{du^k} f^{\varepsilon}\left(\sqrt{\frac{2n}{\lambda}}\right).$ On the other hand, by using the basis functions $a_{n,k}^{\varepsilon}$ defined in \eqref{baseff}, it follows that  these are the only linear obstructions on the values of $ \frac{d^k}{du^k} f^{\varepsilon}\left(\sqrt{\frac{2n}{\lambda}}\right)$. 
\end{remark}

\subsection{Universal optimality of the hexagonal lattice}\label{opin}
In this section, we state our theorem which implies a conjecture of  Cohn,  Kumar,  Miller,  Radchenko and   Viazovska~\cite[Conjecture 7.5]{Maryna3} motivated by the universal optimality of the hexagonal lattice. We begin by stating a conjecture of Cohn and Elkies. This conjecture is based on a version of the linear programing method developed by Cohn and Elkies~\cite{Elkies} for giving upper bounds on the density of the sphere packings in Euclidean spaces. 
\\

Cohn and Elkies conjectured~\cite{Elkies} that there exists a radial Schwartz function $f:\R^2\to \R$ that satisfies 
\begin{enumerate}
\item $f(r)\leq 0$ for $r^2\geq \frac{2}{\sqrt{3}}$,
\item $\mathcal{F}(f)(r)\geq 0$ for all $r$,
\item $f(0)=\mathcal{F}(f)(0),$
\end{enumerate}
where $r=|x|.$
It follows from the Poisson summation formula for the hexagonal lattice that 
\[
f\left(\sqrt{\frac{2n}{\sqrt{3}}}\right)=\mathcal{F}(f)\left(\sqrt{\frac{2n}{\sqrt{3}}}\right)=\frac{d}{dr}\mathcal{F}(f)\left(\sqrt{\frac{2n}{\sqrt{3}}}\right)=0, n\geq 1,
\]
and
\[
\frac{d}{dr}f\left(\sqrt{\frac{2n}{\sqrt{3}}}\right)=0,  n>1,
\]
where $n=x^2+xy+y^2$ for some $x,y\in \mathbb{Z}.$
Constructing a generalized version of this function is equivalent to proving the universality of the hexagonal lattice in the plane~\cite{Maryna3}, which is an outstanding open problem.  Namely, constructing $f_t$ such that 
\begin{enumerate}
\item $f_t(r)\leq e^{-\pi t r^2}$,
\item $\mathcal{F}f_t(r)\geq 0$,
\item $\mathcal{F}f_t(0)-f_t(0)=\theta_0(it)-1.$
\end{enumerate}

Very recetly~Viazovska and her collaborators \cite{Maryna1,Maryna2, Maryna3} in their spectacular works resolved the Cohn and Elkies conjecture in dimensions 8 and 24 and also proved the universality of $E_8$ and the Leech lattice. As discussed by the authors~\cite[page 91]{Maryna3} their method does not generalize directly to dimension $2$ and one needs a new idea for dimension 2 to construct $f$ satisfying the above conditions. We discuss some properties of $f_t.$ It follows from the Poisson summation formula for the hexagonal lattice that 
\[
\begin{split}
f_t\left(\sqrt{\frac{2n}{\sqrt{3}}}\right)=e^{-\frac{2\pi t n}{\sqrt{3}}}, n\geq 1,
\\
\mathcal{F}(f)\left(\sqrt{\frac{2n}{\sqrt{3}}}\right)=\frac{d}{dr}\mathcal{F}(f)\left(\sqrt{\frac{2n}{\sqrt{3}}}\right)=0, n\geq 1,
\end{split}
\]
and
\[
\frac{d}{dr}f\left(\sqrt{\frac{2n}{\sqrt{3}}}\right)=-2\sqrt{\frac{2n}{\sqrt{3}}}\pi t e^{-\frac{2\pi t n}{\sqrt{3}}},  n>1,
\]
where $n=x^2+xy+y^2$ for some $x,y\in \mathbb{Z}.$
The authors in~\cite{Maryna3} showed that in dimensions 8 and 24 the analogues   equations uniquely determine the radial Schwartz function $f$. However, they conjectured~\cite{Maryna3} that in dimension $2,$ the radial Schwartz function $f$ is not uniquely determined by the values of  $f,f',\mathcal{F}(f)$ and  $\mathcal{F}(f)'$ at $\sqrt{\frac{2n}{\sqrt{3}}}$, where $n$ is represented by $x^2+xy+y^2.$ One heuristic reason is that the number of the  integers less than $X$ represented by $x^2+xy+y^2$ is  $X/\sqrt{\log(X)}$ asymptotically. So we have less equations than in dimensions $8$ and $24$, where we have equations associated to each integers $n\geq 0$.  
\\

We prove the conjecture of  Cohn,  Kumar,  Miller,  Radchenko and  Viazovska~\cite[Conjecture 7.5]{Maryna3}.
\begin{theorem}\label{mainconj}
There are infinitely many linearly independent radial Schwartz function $f$ on the plane  such that  $f$ and $\mathcal{F}(f)$ vanishes of order $2$ at   $|x|=\sqrt{\frac{2n}{\sqrt{3}}}$ where $n=x^2+xy+y^2$ for some $x,y \in\Z.$ 

\end{theorem}

\subsubsection{Method of the proof}
In this section we discuss some new ideas that we develop to prove Theorem~\ref{mainconj}. The first step is to cover integers  $n=x^2+xy+y^2$ by a periodic set of integers with small density. Note that this is a basic result in sieve theory and we include a proof  for the convenience  of the reader in Lemma~\ref{Acons}. 
\\

 Suppose that  
\[
A:=\left\{a>100 | a\equiv a_i \mod L, \text{ for some }a_i \text{ where } 1\leq i\leq l\right\}.
\]
Furthermore, we suppose that the density is small  $\frac{l}{L}<\delta.$  We prove a stronger version of  Theorem~\ref{mainconj} that we state next. 
\begin{theorem}\label{strongthm}
Suppose that $\delta<0.001$ and $a\in A$ is any element. 
There exists a radial Schwartz function $f$ such that  $f$ and $\mathcal{F}(f)$ vanishes of order $2$ at $\sqrt{\frac{2m}{\sqrt{3}}}$ where  $m\in A-\{a\},$ and 
\[
f'\left(\sqrt{\frac{2a}{\sqrt{3}}}\right)= 1.
\]
\end{theorem}

We only use Corollary~\ref{mainthm22} for the Hecke triangle group $\Gamma=(2,6,\infty)$ from the previous sections.  Let 
\[
r^{\varepsilon}(\tau;x):=e^{i\pi \tau |x|^2}|^{\varepsilon}_1 (T^{-1}-I)(1+V+\dots+V^5),
\]
and 
\[
s^{\varepsilon}(\tau;x):=e^{i\pi \tau |x|^2}|^{\varepsilon}_1 (1+V+\dots+V^5).
\]

By Corollary~\ref{mainthm22}
\[
r^{\varepsilon}\left(\tau;\sqrt{\frac{2m}{\sqrt{3}}}\right)=0
\]
for every  integer $m\geq0.$ It is easy to check that 
\[
\frac{d}{du}r^{\varepsilon}\left(\tau;\sqrt{\frac{2m}{\sqrt{3}}}\right)=c s^{\varepsilon}\left(\tau;\sqrt{\frac{2m}{\sqrt{3}}}\right),
\]
where $u=|x|^2$ and $c=\pi i.$ Our new idea is to average $r^{\varepsilon}(\tau;x)$ over $\tau$ with respect to a  measure $\mu$ supported on a compact region of the upper half-plane   such that the derivative vanishes at  $\sqrt{\frac{2m}{\sqrt{3}}}$  for every $m\in A-\{a\}$. More precisely, let 
\[
f(x)=\int r^{\varepsilon}(\tau;x) d\mu(\tau).
\]
Then 
\[
f\left(\sqrt{\frac{2m}{\sqrt{3}}}\right)=0, \text{ and } \frac{d}{du}f\left(\sqrt{\frac{2m}{\sqrt{3}}}\right)=c\int s^{\varepsilon}(\tau;x) d\mu(\tau).
\]
We construct $\mu$  as the weak$^*$ limit of a sequence of  measures $\{\mu_n\}$ such that 
\[
\int s^{\varepsilon}\left(\tau;\sqrt{\frac{2m}{\sqrt{3}}}\right) d\mu_n(\tau)=0
\]
where $m\in A-\{a\}$ and $0\leq m<n.$
The existence of a  weak$^*$ limit is a consequence of the compactness of the space of probability measures  on a compact Borel measure space. By computing the first derivative at $\sqrt{\frac{2a}{\sqrt{3}}}$ we show that $f(x)\neq 0.$ Constructing $\mu_n$ is  challenging and is at the heart of our proof that we discuss in Section~\ref{zeros}. In particular, $\mu_n$ is a probability measure with a support on $\X_{\delta},$ where 
\[
\X_{\delta}:= \left\{\frac{\sqrt{3}}{2}+x+i0.27: |x|<\delta\right\}.
\]

\section{Modular forms for the Hecke triangle group}\label{hecket}
We discuss  the Hecke triangle group $(2,l,\infty)$ and give an explicit  basis for the associated space of modular forms.  We refer the reader to the excellent book of Berndt and Knopp~\cite{Berndt}, and the thesis of Jonas Jermann~\cite{Jonas}.  We only prove some results that we could not  find in the literature.
\subsection{Hecke triangle group $(2,l,\infty)$}
Let 
\[
S:=\begin{bmatrix}  0 & 1 \\-1 &0
\end{bmatrix}  , \text{ and }T:=\begin{bmatrix}  1 & \lambda \\0 &1
\end{bmatrix} \text{ and } U:=ST.
\]
We consider the action of the fractional transformation on the upper half-plane. The group generated by $S$ and $T$ is the  triangle group $(2,l,\infty).$
It is also well-known that as an abstract group 
\[
\Gamma= \frac{\Z}{2\Z}*\frac{\Z}{l\Z},
\]
where $\frac{\Z}{2\Z}*\frac{\Z}{l\Z}$ is the free product of the cyclic groups and  the isomorphism is given by sending $S$ to the generator of $ \frac{\Z}{2\Z}$ and $U$ to a generator of $\frac{\Z}{l\Z}$.

\subsection{Modular forms} We record some well-known facts about the space of holomorphic modular forms for $\Gamma$; see~\cite[Chapter 5]{Berndt}  for a detailed discussion. For $z\in \CC$, $z\neq 0,$ and $r\in R$  let
\[
z^r= |z|^r e^{ir\arg(z)},
\]
where $-\pi\leq\arg(z)<\pi.$ 
Let $M_r^{\varepsilon}(\Gamma)$ be the space of holomorphic bounded functions $f(z)$ defined on the upper half-plane which satisfy 
\[
\begin{split}
f(z)=\varepsilon \left(\frac{ i}{z}\right)^r f\left(\frac{-1}{z}\right) ,
\\
f(z)=f\left(z+\lambda\right),
\end{split}
\] 
where $\varepsilon=\pm1.$ For $\gamma\in \Gamma$ and $f\in M_r^{\varepsilon}(\Gamma)$ let 
\[
j_{r}^{\varepsilon}(z,\gamma):=\frac{f(z)}{f(\gamma z)},
\]
where $f(\gamma z)\neq 0.$ It follows that 
\[
|j_{r}^{\varepsilon}(z,\gamma)|=\frac{1}{|cz+d|^r},
\]
where $\gamma=\begin{bmatrix} a & b \\ c & d  \end{bmatrix}.$ We define the slash operator acting on $h\in C(\HH)$ as
\[
h|_r^{\varepsilon}\gamma= j_{r}^{\varepsilon}(z,\gamma)h(\gamma z).
\]
We cite~\cite[Theorem 5.5]{Berndt} and~\cite[Theorem 5.6]{Berndt}.
\begin{theorem}[Theorem 5.5 of \cite{{Berndt}}]\label{zerobr}
There exist modular forms $f_{w}\in M_{\frac{4}{l-2}}^1(\Gamma),$ $f_i \in M_{\frac{2l}{l-2}}^{-1}(\Gamma)$ and $f_{\infty}\in M_{\frac{4l}{l-2}}^1(\Gamma)$ such that each has a simple root at $w_1$, $i$ and $i\infty$, respectively, and no other zeros. 
\end{theorem}
\begin{theorem}[Theorem 5.6 of \cite{Berndt}]\label{dimbr}Suppose that $\dim M_{l}^{\varepsilon}(\Gamma)\neq0.$ Then
\[
l=\frac{4m}{l-2}+1-\varepsilon,
\]
where $m\geq 1$ is an integer, and 
\[
\dim M_{l}^{\varepsilon}(\Gamma)= 1+\lfloor \frac{m+(\varepsilon-1)/2}{l} \rfloor.
\]
\end{theorem}

\subsection{Weakly holomorphic modular forms}\label{basisk}
We remove the boundedness condition in the upper half-plane, and allow a pole at the cusp at $\infty$ and obtain weakly holomorphic modular forms; see~\cite[Chapter 3]{Jonas} for a detailed discussion. We denote the space of weakly holomorphic modular forms of weight $r$ and multiplier $\varepsilon$ by 
\( 
M^{\varepsilon !}_r(\Gamma).
\)
Let $J(z)$ be Hauptmodul function for $\Gamma$. It is the unique Riemann map from the hyperbolic triangle with vertices at $w_2$, $i$ and $i\infty$ to the upper half-plane normalized such that
\[
J(w_2)=0, 
\\
J(i\infty)=\infty,
\\
J(q)=q^{-1}+O(1),
\]
where $q=e^{\frac{2\pi i z}{\lambda}}.$ 
It is well known that $J(z)\in M_0^{+!},$ and has rational coefficients~\cite{Lehner}.  
\\

Let $n_p(f)$ denote the order of vanishing of the meromorphic function $f$ at point $l.$ We write $N(f)$ for the sum of orders of  points except from $\{ i , w,i\infty  \}$.
We cite~\cite[Lemma 3.1]{Jonas} which extends  \cite[Lemma 5.1]{Berndt} to weakly holomorphic modular forms.  
\begin{lemma}[Lemma 3.1 of \cite{Jonas}]
Suppose that $f\in M_r^{\varepsilon!}$ and $f\neq 0.$ Then
\[
N(f)+n_{\infty}(f)+\frac{1}{2}n_i(f)+\frac{n_{w}(f)}{r}=\frac{r(r-2)}{4r}.
\]
\end{lemma}
 Note that 
 \[
 N(J)=0, n_{\infty}(J)=-1, n_i(J)=0, \text{ and } n_w(J)=l. 
 \]
 The derivative of $J$ is also a weakly holomorphic modular form. We have 
\[
J'(z) \in M_{2}^{-!}(\Gamma),
\]
and 
 \begin{equation}\label{jjj}
 N(J')=0, n_{\infty}(J')=-1, n_i(J')=1, \text{ and } n_w(J')=l-1. 
 \end{equation}

\subsubsection{Weakly holomorphic modular forms of weight $2k+1$} Suppose that $d(\varepsilon,k)=\dim M_{2k+1}^{-\varepsilon}(\Gamma),$ where $k\in\Z$ and $k\geq 0.$ We write $d$ for $d(\varepsilon,k)$ in this section. 
\begin{lemma}
Suppose that $d>0$, then there exists a unique $f_{d-1,k}\in M_{2k+1}^{-\varepsilon}(\Gamma)$ such that 
\[
f_{d-1,k}=q^{d-1}+\alpha_{d-1}q^{d}+O(q^{d+1}).
\]
Furthermore, $d=0$ if and only if $\varepsilon=-1$ and $k=0.$ In this case, there exists a unique $f_{-1,0}\in M_{1}^{-!}(\Gamma)$ such that
\[
f_{-1,0}=q^{-1}+\alpha_{-1}+O(q).
\]

\end{lemma}
\begin{proof}
Suppose that $d>0.$  By linear algebra there exists  $f\in M_{2k+1}^{-\varepsilon}(\Gamma)$ such that 
\(
f=O(q^{l}),
\)
where $l\geq d-1.$
It follows that $\{f, fJ,\dots,fJ^{l}\}\subset M_{2k+1}^{-\varepsilon}(\Gamma)$ are linearly independent. Hence, there exists $f_{d-1,k}\in M_{2k+1}^{-\varepsilon}(\Gamma)$ such that 
\[
f_{d-1,k}=q^{d-1}+\alpha_{d-1}q^{d}+O(q^{d+1}).
\]
For $ 0\leq i\leq d-1$, let
\[
f_{i,k}:=f_{d-1,k}P_i(J(z))=q^{i}+\alpha_iq^{d}+O(q^{d+1})\in M_{2k+1}^{-\varepsilon}(\Gamma)
\]
for some real polynomial $l_i(x).$
It is easy to check that $\{f_{i,k}\}$ form a basis for $M_{2k+1}^{-\varepsilon}(\Gamma),$ and $l=d-1.$
\\

Next suppose that $d=0.$
 By Theorem~\ref{dimbr} and our assumption $l\equiv 2$ mod $4$, we have $d>0$ unless $k=0$ and $\varepsilon=1.$
Then $k=0$ and $\varepsilon=-1.$ Let
\[
f_{-1,0}:=\frac{-\frac{\lambda}{2\pi i}J'(z)}{f_{w}^{\frac{l-2}{4}}}.
\]
Note that the order of vanishing of $J'$ at $w_1$ is $l-1> \frac{l-2}{4}$, hence $f_{-1,0}\in M_{1}^{-!}(\Gamma).$ We have
\[
f_{-1,0}=q^{-1}+\alpha_{-1}+O(q).
\]
\end{proof}
Next, we introduce a duality between $M_{2k+1}^{-\varepsilon!}(\Gamma)$ and $ M_{-2k+1 }^{\varepsilon!}(\Gamma).$ 
\begin{lemma}\label{duality}
Suppose that $f=\sum_{n} a_nq^n \in M_{2k+1}^{-\varepsilon!}(\Gamma)$ and $\phi=\sum b_mq^m \in M_{-2k+1 }^{\varepsilon!}(\Gamma).$ Then,
\[
\sum_{n} a_n b_{-n}=0.
\]
\end{lemma}
\begin{proof}
We have 
\[
\sum_{n} a_n b_{-n}= \frac{1}{\lambda}\int_{w_1}^{w_2}f(z)g(z) dz. 
\]
We note that $f(z)g(z)\in M_{2!}^{-}(\Gamma).$ We show that 
\[
\int_{w_1}^{w_2}h(z) dz=0
\]
for every $h(z)\in M_{2!}^{-}(\Gamma).$ 
By substituting $w=\frac{-1}{z}$ in the integral, we obtain 
\[
\begin{split}
\int_{w_1}^{w_2}h(z) dz&=\int_{w_2}^{w_1}  h\left(\frac{-1}{w}\right) d (\frac{-1}{w})
\\
&=\int_{w_2}^{w_1}  h\left(\frac{-1}{w}\right) \frac{1}{w^2} dw
\\
&=\int_{w_2}^{w_1} h(w) dw= -\int_{w_1}^{w_2}h(z) dz,
\end{split}
\]
where we used 
\(
h(w)= \frac{ 1}{w^2} h\left(\frac{-1}{w}\right).
\) This implies that 
\[
\int_{w_1}^{w_2}h(z) dz=0. 
\]
\end{proof}
Let 
 \begin{equation}\label{jder}
 \phi_{d,k}^{\varepsilon}(z):=\frac{\lambda}{2\pi i} \frac{J'(z)}{f_{d-1,k}} 
 \end{equation}
It follows form~\eqref{jjj} that  $ \phi_{d,k}\in M_{-2k+1}^{\varepsilon !}(\Gamma).$ Moreover, by Lemma~\ref{duality} $\phi_{d,k}^{\varepsilon}(z)$   has the following coefficients
\[
\phi_{d,k}^{\varepsilon}(z)=q^{-d}-\sum_{i=0}^{d-1} \alpha_i q^{-i}+O(q).
\]

Let $Q(x)$ be a polynomial with complex coefficients. Define
\[
\mathcal{Q}^{\varepsilon}(z):=Q(J^{+}(q)) \phi_{d,k}^{\varepsilon}(z).
\]
It is easy to check that $\mathcal{Q}^{\varepsilon}_k(z)$ satisfies the modular transformation of a weight $-2k+1$ modular form with multiplier $\varepsilon$; see \cite[Proposition 3.13]{Jonas}. Moreover,
\[\mathcal{Q}^{\varepsilon}(z)= O(q^{-(\deg(Q)+d)}), \text{ as }z\to \infty.\]

Let   
\begin{equation}\label{defeqqq}
\phi_{n,k}^{\varepsilon}:=Q^{\varepsilon}_{n,k}(J)\phi_{d,k}^{\varepsilon}(z)= q^{-n}+O(q^{-d(\varepsilon,k)+1}), n\geq d(\varepsilon,k),
\end{equation}
be the unique modular forms of weight $-2k+1$ with respect to $\Gamma.$
Let 
\[
K_k^{\varepsilon}(z,\tau):=\sum_{n\geq d(\varepsilon,k)}^{\infty} \phi_{n,k}^{\varepsilon}(z)  e^{\frac{2\pi i \tau n}{\lambda}}.
\]
Zagier proved an explicit formula for a special case of the above generating series for the modular curve  in~\cite{Zagiertr}. 
The following Proposition is stated for $d(\varepsilon,k)>0$ in ~\cite[Theorem 3.14]{Jonas}. We  write a proof for $d(\varepsilon,k)\geq 0 $ that follows  closely the proof of~\cite[Theorem 2]{Jenkins}.
\begin{proposition}\label{genf}
We have
\[
K_k^{\varepsilon}(z,\tau)=\frac{\phi_{d,k}^{\varepsilon}(z)f_{d-1,k}^{-\varepsilon}(\tau)}{J(\tau)-J(z)}.
\]

In particular, $K_k^{\varepsilon}(z,\tau)$ is a modular form of weight $2k+1$ and multiplier $-\varepsilon$ with respect to $\tau$ and of weight $1-2k$ and multiplier $\varepsilon$ with respect to $z.$
\end{proposition}
\begin{proof}
 By~\eqref{jder}, it is enough to show that 
\[
 \frac{1}{\lambda}K_k^{\varepsilon}(z,\tau)=-\frac{1}{2\pi i}\frac{J'(\tau)\phi_{d,k}^{\varepsilon}(z)}{\left(J(\tau)-J(z)\right)\phi_{d,k}^{\varepsilon}(\tau)}.
\]

 For $\Im\tau>\Im z,$ the generating series is convergent and by circle method, 
\[
\phi_{n,k}^{\varepsilon}(z)= \frac{1}{\lambda}\int_{\tau_0}^{\tau_0+\lambda} K_k^{\varepsilon}(z,\tau)q(-n\tau) d\tau,
\]
where $\Im(\tau_0)>\Im(z).$
It is enough to show that for $n\geq d$
\[
\phi_{n,k}^{\varepsilon}(z)=-\frac{1}{2\pi i}\int_{\tau_0}^{\tau_0+\lambda}\frac{J'(\tau)\phi_{d,k}^{\varepsilon}(z)}{\left(J(\tau)-J(z)\right)\phi_{d,k}^{\varepsilon}(\tau)}q(-n\tau) d\tau.
\]
Since  $f_{d-1,k}^{-\varepsilon}(\tau)q(-n\tau) $ is holomorphic at cusp $\infty$ for every $n< d$ and  by \eqref{jder},  we have 
\[
\int_{\tau_0}^{\tau_0+\lambda} \frac{J'(\tau)\phi_{d,k}^{\varepsilon}(z)}{\left(J(\tau)-J(z)\right)\phi_{d,k}^{\varepsilon}(\tau)}q(-n\tau) d\tau=\int_{\tau_0}^{\tau_0+\lambda} \frac{J'(z)f_{d-1,k}^{-\varepsilon}(\tau)}{\left(J(\tau)-J(z)\right)f_{d-1,k}^{-\varepsilon}(z)}q(-n\tau) d\tau=0.
\]
For $n\geq d$, we note that
\[
\phi_{n,k}^{\varepsilon}(\tau)-q(-n\tau)=O(q(-(d-1)\tau)). 
\]
Hence, 
\[
\int_{\tau_0}^{\tau_0+\lambda} \frac{J'(\tau)\phi_{d,k}^{\varepsilon}(z)}{\left(J(\tau)-J(z)\right)\phi_{d,k}^{\varepsilon}(\tau)}\left(\phi_{n,k}^{\varepsilon}(\tau)-q(-n\tau)\right) d\tau=0.
\]
Therefore, 
\[
\begin{split}
-\frac{1}{2\pi i}\int_{\tau_0}^{\tau_0+\lambda} \frac{J'(\tau)\phi_{d,k}^{\varepsilon}(z)}{\left(J(\tau)-J(z)\right)\phi_{d,k}^{\varepsilon}(\tau)}q(-n\tau) d\tau&=-\frac{1}{2\pi i}\int_{\tau_0}^{\tau_0+\lambda} \frac{J'(\tau)\phi_{d,k}^{\varepsilon}(z)}{\left(J(\tau)-J(z)\right)\phi_{d,k}^{\varepsilon}(\tau)}\phi_{n,k}^{\varepsilon}(\tau)d\tau
\\
&=\phi_{d,k}^{\varepsilon}(z)\frac{-1}{2\pi i}\int_{J(\tau_0)}^{J(\tau_0+\lambda)} \frac{\mathcal{Q}^{\varepsilon}_{n,k}(J(\tau) )}{\left(J(\tau)-J(z)\right)}dJ(\tau)
\\
&=\phi_{d,k}^{\varepsilon}(z)\mathcal{Q}^{\varepsilon}_{n,k}(J(z) )  = \phi_{n,k}^{\varepsilon}(z).
\end{split}
\]
\end{proof}
\begin{remark}
Note that there is a duality between  weights $-2k+1$ and  $2k+1.$ More precisely, 
for every $k\geq 0,$ we have 
\[
\sum_{n\geq d(\varepsilon,k)}^{\infty} \phi_{n,k}^{\varepsilon}(z)  e^{2\pi i \tau \frac{n}{\lambda}}=\frac{\phi_{d,k}^{\varepsilon}(z)f_{d-1,k}^{-\varepsilon}(\tau)}{\left(J(\tau)-J(z)\right)}= -\sum_{n\geq 1-d(\varepsilon,k)}^{\infty} f_{n,k}^{-\varepsilon}(\tau)  e^{2\pi i z \frac{n}{\lambda}}.
\]
\end{remark}

\subsection{Modular integrals} 
For the  $k$-th derivatives interpolation formula, we introduce  $F_k^{\varepsilon}(\tau,x)$ which is holomorphic for $\tau>\sin\left(\frac{\pi}{l}\right)$ (it has multiple values with analytic continuation for $\sin\left(\frac{\pi}{l}\right) \geq\Im(\tau)>0$) such that
\begin{equation}\label{cobn}
\begin{split}
F_k^{\varepsilon}(\tau,x)|_{2k+1}{I+\varepsilon S}=\frac{e^{\pi i \tau x^2}}{(i\pi \tau)^k} |_{2k+1}{I+\varepsilon S},
\\
F_k^{\varepsilon}(\tau+\lambda,x)=F_k^{\varepsilon}(\tau,x),
\end{split}
\end{equation}
where $\Im (\tau),\Im(\frac{-1}{\tau})>\sin\left(\frac{\pi}{l}\right).$
Following Knopp~\cite{Knopp,Knoppm}, we denote the solution to the above functional equations by Modular integrals, and the given function on the  right hand side by the period function. We find a solution to the above functional equations in section~\ref{modint}. Our method is based on the work Radchenko and   Viazovska~\cite{inter} which follows closely the work of Duke, Imamo\={g}lu, and  T\'{o}th~\cite{Dukec}.

\section{Interpolation basis}
\subsection{Interpolation basis for higher derivatives}
We define a family of   eigenfunctions of the Fourier transformation with eigenvalue $\varepsilon$ such that their $k$-derivatives  vanish with order $1$ on all except one point of the form
$\sqrt{\frac{2n}{\lambda}}$ where $n\geq d(\varepsilon,k)=\dim M_{2k+1}^{-\varepsilon}(\Gamma).$ Recall that
\[
a_{n,k}^{\varepsilon}(x):=   \frac{1}{\lambda}\int_{w_1}^{w_2}\phi_{n,k}^{\varepsilon}(z)\frac{e^{\pi i z |x|^2}}{(iz\pi)^k}  dz.
\]
\begin{lemma}
We have 
\[
\widehat{a_{n,k}^{\varepsilon}}(\xi)=\varepsilon a_{n,k}^{\varepsilon}(\xi).
\]
\end{lemma}
\begin{proof} We have 
\[
\begin{split}
\widehat{a_{n,k}^{\varepsilon}}(\xi)
&=  \int_{\R^2}  \frac{1}{ \lambda}\int_{w_1}^{w_2}\phi_{n,k}^{\varepsilon}(z)\frac{e^{\pi i z |x|^2}}{(iz\pi)^k}   e^{2 \pi i \left<x,\xi \right>}dz dx
\\
 &= \frac{1}{(i\pi)^k \lambda}\int _{w1}^{w_2}\phi_{n,k}^{\varepsilon}(z)\frac{i}{z^{k+1}} e^{\pi i \frac{-1}{z} |\xi|^2} dz 
 \\
  &= \frac{1}{(i\pi)^k \lambda}\int _{w2}^{w_1} \phi_{n,k}^{\varepsilon}\left(\frac{-1}{v}\right)i(-v)^{k+1} e^{\pi i v |\xi|^2} v^{-2} dv 
  \\
    &= \frac{1}{(i\pi)^k\lambda}\int _{w2}^{w_1} -\phi_{n,k}^{\varepsilon}\left(\frac{-1}{v}\right)\left(\frac{i}{v}\right)^{-2k+1} \frac{e^{\pi i v |\xi|^2}}{v^k}  dv 
  \\
  &=  \frac{\varepsilon}{\lambda}\int _{w2}^{w_1} -\varepsilon\phi_{n,k}^{\varepsilon}(v)  \frac{e^{\pi i v |\xi|^2}}{(i\pi v)^k}  dv =\varepsilon a_{n,k}^{\varepsilon}(\xi).
\end{split}
\]
where $v=\frac{-1}{z}.$
\end{proof}

Recall that $u=|x|^2.$
\begin{lemma} For $m,n \geq d(\varepsilon,k)$, we have 
\[
\frac{d^k}{du^k} a_{n,k}^{\varepsilon}\left(\sqrt{\frac{2m}{\lambda}}\right)
=\delta(m,n).
\]
\end{lemma}
\begin{proof} 
We have 
\[
\begin{split}
\frac{d^k}{du^k} a_{n,k}^{\varepsilon}\left(x\right)&= \frac{1}{\lambda} \int _{w_1}^{w_2} e^{-\frac{2\pi iz}{\lambda} n} \left(\frac{d}{d|x|^2}\right)^k\frac{e^{\pi i z |x|^2}}{(iz\pi)^k} dz +  \frac{1}{\lambda} \int _{w_1}^{w_2}\sum_{j\geq 0}r_{n,k}^{\varepsilon}(j)e^{\frac{2\pi iz}{\lambda}j}  \frac{d}{d|x|^2}\frac{e^{\pi i z |x|^2}}{(iz\pi)^k} dz.
\\
&= \frac{1}{\lambda}\int_{w_1}^{w_2}e^{2\pi i z(\frac{-n+\lambda|x|^2/2}{\lambda})}dz +\sum_{j\geq 0}\frac{r_{n,k}^{\varepsilon}(j)}{2\pi i (j+\lambda|x|^2/2)} \Big|_{w_1}^{w_2}e^{2\pi i z(\frac{j+\lambda|x|^2/2}{\lambda})} 
\end{split}
\]
Suppose that  $x=\sqrt{\frac{2m}{\lambda}}$, then  we have 
\[
\begin{split}
\frac{d^k}{du^k} a_{n,k}^{\varepsilon}\left(\sqrt{\frac{2m}{\lambda}}\right)&=
 \delta_{m,n}+\sum_{j\geq 0}\frac{r_{n,k}^{\varepsilon}(j)}{2\pi i (j+m)} \Big|_{w_1}^{w_2}e^{2\pi i z(\frac{j+m}{\lambda})} 
\\
&= \delta_{m,n}+\sum_{m\geq 0}\frac{r_{n,k}^{\varepsilon}(j)}{2\pi i (j+m)} e^{-\pi(\frac{k+m}{\lambda})}  \sin \left( \pi  (k+m)  \right).
\end{split}
\]
Since $ \sin \left( \pi  (m+ \lambda\frac{x^2}{2})  \right)=0,$
\[ \frac{d^k}{du^k} a_{n,k}^{\varepsilon}\left(\sqrt{\frac{2m}{\lambda}}\right)=
 \delta_{m,n}.\]
\end{proof}

\section{Modular integral with given period function}~\label{modint}

\begin{figure}
\centering
\begin{tikzpicture}[scale=3]
    \draw[-latex](\myxlow,0) -- (\myxhigh ,0);
    \pgfmathsetmacro{\succofmyxlow}{\myxlow+1}
    \foreach \x in {\myxlow,\succofmyxlow,...,\myxhigh}
    {   \draw (\x,-0) -- (\x,-0.0) node[below,font=\tiny] {\x};
    }
        {   \draw (-0.866,-0.1)node[below,font=\tiny]{$-\frac{\lambda}{2}$}--(-0.866,0.5)node[left,font=\tiny]{$w_1$} (0.866,0.5)node[right,font=\tiny]{$w_2$} -- (0.866,-0.1) node[below,font=\tiny] {$\frac{\lambda}{2}$} ;}{
         \draw (-0.866,1.2)node[left,font=\tiny]{$\tau_0$}--(0.866,1.2)node[right,font=\tiny]{$\tau_0+\lambda$}  ;
    }
    {
         \draw(0.866,1.2)node[below,font=\tiny]{$\Omega_0$}  ;
    }

    \foreach \y  in {0,1}
    {   \draw (0,\y) -- (-0.0,\y) node[left,font=\tiny] {\pgfmathprintnumber{\y}};
    }
    \draw[-latex](0,-0.0) -- (0,1.6);
    \begin{scope}   
        \clip (\myxlow,0) rectangle (\myxhigh,2.2);
            {   \draw[very thin, blue] (1,0) arc(0:180:1);
            }   
            
            {\draw [very thin, red] (-0.866,0.5)node[left,font=\tiny]{$w_1$}--(0.866,0.5)node[right,font=\tiny]{$w_2$} (0,0.5)node[right,font=\tiny]{$\gamma$} (0,0.8)node[right,font=\tiny]{$D_1$} ;
            }
            
                      {   \draw (1,1)[very thin, red] arc(0:360:1) (0,2)node{$S\gamma$};
            }

    \end{scope}
    \begin{scope}
            \begin{pgfonlayer}{background}
            \clip (-0.866,0) rectangle (0.866,1.7);
            \clip   (1,1.7) -| (-1,0) arc (180:0:1) -- cycle;
            \fill[gray,opacity=0.8] (-1,-1) rectangle (1,2);
        \end{pgfonlayer}
    \end{scope}
\end{tikzpicture}
\captionof{figure}{}\label{heckefff}
\end{figure}

\begin{lemma} \label{inbd}Suppose that  $z$ is on the arc from $w_1$ to $w_2.$ 
We have
\[
|\phi_{n,k}^{\varepsilon}(z)|\ll_{k} \delta^{-1}e^{ \frac{2\pi (1+\delta)}{\lambda}n}.
\]
for any $\delta>0.$
\end{lemma}
\begin{proof} Recall that 
\[
K_k^{\varepsilon}(z,\tau):=\sum_{n\geq d(\varepsilon,k)}^{\infty} \phi_{n,k}^{\varepsilon}(z)  e^{2\pi i \tau \frac{n}{\lambda}}.
\]

By Proposition~\ref{genf}, we have  
\[
K_k^{\varepsilon}(z,\tau)=\frac{\phi_{d,k}^{\varepsilon}(z)f_{d-1,k}^{-\varepsilon}(\tau)}{J(\tau)-J(z)} .\]
By circle method, we have 
\[
\begin{split}
|\phi_{n,k}^{\varepsilon}(z)|&\ll \left|\int_{\tau_0}^{\tau_0+\lambda} K_k^{\varepsilon}(z,\tau) e^{-2\pi i \tau \frac{n}{\lambda}}d\tau\right|
\\
&\ll \int_{\tau_0}^{\tau_0+\lambda}  \left| \frac{\phi_{d,k}^{\varepsilon}(z)f_{d-1,k}^{-\varepsilon}(\tau)}{J(\tau)-J(z)}  e^{-2\pi i \tau \frac{n}{\lambda}}d\tau   \right|,
\end{split}
\]
where $\tau_0=-\frac{\lambda}{2}+(1+\delta)i$ for some fixed $0<\delta<1;$ see Figure~\ref{heckefff}. We note that $J(z)$ is injective on the fundamental domain of $\Gamma.$ Hence 
\[
 \frac{1}{|J(\tau)-J(z) | } = O(\delta^{-1}),   
\]
where $\Im(\tau)=1+\delta$ and, $z$ belongs to the arc between $w_1$ and $w_2.$ Since $\phi_{d,k}^{\varepsilon}(z)$ and $f_{d-1,k}^{-\varepsilon}(\tau)$ are fixed and holomorphic on the upper half-plane, we have 
\[
|\phi_{d,k}^{\varepsilon}(z)f_{d-1,k}^{-\varepsilon}(\tau) |=O_{k}(1),
\]
where $\Im(\tau)=1+\delta$ and $z$ belongs to the arc between $w_1$ and $w_2.$ Therefore,
\[
|\phi_{n,k}^{\varepsilon}(z)| \ll_{k} \int_{\tau_0}^{\tau_0+\lambda}  \left| \delta^{-1} e^{-2\pi i \tau \frac{n}{\lambda}}d\tau   \right| \ll  \delta^{-1}e^{ \frac{2\pi (1+\delta)}{\lambda}n}.
\]
\end{proof}
\begin{lemma}\label{inibd}
We have 
\[
|a_{n,k}^{\varepsilon}(x)| \ll_k  \delta^{-1}e^{ \frac{2\pi (1+\delta)}{\lambda}n}e^{-\pi\sin\left(\frac{\pi}{l}\right) |x|^2}
\]
for any $\delta>0.$
\end{lemma}
\begin{proof}
By Lemma~\ref{inbd}, we have 

\[
|a_{n,k}^{\varepsilon}(x)|=  \frac{1}{\lambda} \left|\int_{w_1}^{w_2}\phi_{n,k}^{\varepsilon}(z)\frac{e^{\pi i z |x|^2}}{(iz\pi)^k}  dz \right|\ll_{k}   \delta^{-1}e^{ \frac{2\pi (1+\delta)}{\lambda}n}  e^{-\pi\sin\left(\frac{\pi}{l}\right) |x|^2} .
\]
\end{proof}

Lemma~\ref{inibd} implies that the following series is absolutely convergent for $\Im(\tau)>1$
\[
F_k^{\varepsilon}(\tau,x):=\sum_{n\geq 0}a_{n,k}^{\varepsilon}(x) e^{2\pi i \tau \frac{n}{\lambda}}.
\]
Moreover, we can switch the order of summation and the integration and obtain for $\Im(\tau)>1$
\[
F_k^{\varepsilon}(\tau,x)=\frac{1}{\lambda}\int _{w1}^{w_2} \left( \sum_{n\geq 0} \phi_{n,k}^{\varepsilon}(z)e^{2\pi i \tau \frac{n}{\lambda}} \right)  \frac{e^{\pi i z |x|^2}}{(iz\pi)^k}   dz.
\]
By Proposition~\ref{genf}, we have 
\[
\sum_{n\geq 0} \phi_{n,k}^{\varepsilon}(z)e^{2\pi i \tau \frac{n}{\lambda}}  =K_k^{\varepsilon}(z,\tau)=\frac{\phi_{d,k}^{\varepsilon}(z)f_{d-1,k}^{-\varepsilon}(\tau)}{J(\tau)-J(z)}.
\]
Therefore,  we have 
\begin{equation}\label{contourdef}
F_k^{\varepsilon}(\tau,x)=\frac{1}{\lambda}\int _{w1}^{w_2}\frac{\phi_{d,k}^{\varepsilon}(z)f_{d-1,k}^{-\varepsilon}(\tau)}{\left(J(\tau)-J(z)\right)} \frac{e^{\pi i z |x|^2}}{(iz\pi)^k} dz.
\end{equation}
for $\Im(\tau)>1.$
\begin{proposition}\label{ccontour}
$F^{\varepsilon}(x,\tau)$ has analytic continuation to $\Im(\tau)>\sin\left(\frac{\pi}{l}\right)$. Moreover, we have 
\[
\begin{split}
F_k^{\varepsilon}(\tau,x)|_{2k+1}{I+\varepsilon S}=\frac{e^{\pi i \tau x^2}}{(i\pi\tau)^k} |_{2k+1}{I+\varepsilon S},
\\
F_k^{\varepsilon}(\tau+\lambda,x)=F_k^{\varepsilon}(\tau,x),
\end{split}
\]
when both sides of the above identities are in the domain of definition of $F_k^{\varepsilon}(\tau,x).$
\end{proposition}
\begin{proof}
We note that $J(z)$, takes real values on the unit circle and the image of the arc between $w_1$ and $w_2$ is the interval $[0,J(i)]\subset \R.$
We also note that the contour integral in~\eqref{contourdef} is well defined for every $\tau$ such that 
\[
j(\tau)\notin [0,J(i)]\subset\R.
\]
This implies that $F^{\varepsilon}(x,\tau)$ has analytic continuation to the Fundamental domain $\mathcal{D}$ and all its horizontal translations by $z\to z\pm\lambda$. We denote this region by $\Omega_0$;  see Figure~\ref{heckefff}.
\\

Let $\gamma$ be the chord between $w_1$ and $w_2$, and $D_1$ be the region between the chord and the arc on the unit circle between $w_1$ and $w_2.$ Let $\Omega_1$ be the union of $D_1$ and all its horizontal translations by $z\to z\pm\lambda$; see Figure~\ref{heckefff}.  Next, we analytically continue $F^{\varepsilon}(x,\tau)$ on $D_1.$ It follows from $F^{\varepsilon}(x,\tau)=F^{\varepsilon}(x,\tau+\lambda)$   that it has analytic continuation to $\Omega_1.$
\\

Let $S\gamma$ be the image of $\gamma$ by sending $z \to \frac{-1}{z}.$   It is easy to check that $S\gamma\subset  \Omega_0.$ Suppose that $\tau \in D_1,$ then $S\tau\in \Omega_0.$ Moreover, we have
\begin{equation}\label{eqdiff}
\begin{split}
\frac{1}{\lambda}\int_{\gamma}\frac{\phi_{d,k}^{\varepsilon}(z)f_{d-1,k}^{-\varepsilon}(\tau)}{\left(J(\tau)-J(z)\right)} \frac{e^{\pi i z |x|^2}}{(iz\pi)^k} dz -\frac{1}{\lambda}&\int _{w1}^{w_2}\frac{\phi_{d,k}^{\varepsilon}(z)f_{d-1,k}^{-\varepsilon}(\tau)}{\left(J(\tau)-J(z)\right)} \frac{e^{\pi i z |x|^2}}{(iz\pi)^k} dz
\\
 &=\frac{2\pi i}{\lambda}Res_{z=\tau}\left(\frac{\phi_{d,k}^{\varepsilon}(z)f_{d-1,k}^{-\varepsilon}(\tau)}{\left(J(\tau)-J(z)\right)} \frac{e^{\pi i z |x|^2}}{(iz\pi)^k}\right)
\\
&= \frac{2\pi i}{\lambda}\frac{\phi_{d,k}^{\varepsilon}(\tau)f_{d-1,k}^{-\varepsilon}(\tau)}{-J'(\tau)}\frac{e^{\pi i \tau |x|^2}}{(i\tau\pi)^k}
\\
&=\frac{e^{\pi i \tau |x|^2}}{(i\tau\pi)^k},
\end{split}
 \end{equation}
where we use  identity~\eqref{jder}. Similarly, suppose that $\tau\in S(D_1)\subset \Omega_0$ then
\[
\begin{split}
\frac{1}{\lambda}\int_{\gamma} \frac{\phi_{d,k}^{\varepsilon}(z)f_{d-1,k}^{-\varepsilon}(\tau)}{\left(J(\tau)-J(z)\right)} \frac{e^{\pi i z |x|^2}}{(iz\pi)^k} dz -\frac{1}{\lambda}&\int_{w1}^{w_2}\frac{\phi_{d,k}^{\varepsilon}(z)f_{d-1,k}^{-\varepsilon}(\tau)}{\left(J(\tau)-J(z)\right)} \frac{e^{\pi i z |x|^2}}{(iz\pi)^k} dz
\\
&=\frac{2\pi i}{\lambda}Res_{z=\frac{-1}{\tau}}\left(\frac{\phi_{d,k}^{\varepsilon}(z)f_{d-1,k}^{-\varepsilon}(\tau)}{\left(J(\tau)-J(z)\right)} \frac{e^{\pi i z |x|^2}}{(iz\pi)^k}\right)
\\
&= \frac{2\pi i}{\lambda} \frac{\phi_{d,k}^{\varepsilon}(\frac{-1}{\tau})f_{d-1,k}^{-\varepsilon}(\tau)}{-J'(\frac{-1}{\tau})}\frac{e^{\pi i \frac{-1}{\tau} |x|^2}}{(\frac{-i}{\tau}\pi)^k}
\\
&=-\varepsilon\left(\frac{i}{\tau}\right)^{2k+1}  \frac{e^{\pi i \frac{-1}{\tau} |x|^2}}{(\frac{-i}{\tau}\pi)^k}.
\end{split}
 \]
We note that for $\tau\in \mathcal{D}-SD_1,$ we have 
\[
F^{\varepsilon}(x,\tau)= \frac{1}{\lambda}\int _{w1}^{w_2}\frac{\phi_{d,k}^{\varepsilon}(z)f_{d-1,k}^{-\varepsilon}(\tau)}{\left(J(\tau)-J(z)\right)} \frac{e^{\pi i z |x|^2}}{(iz\pi)^k} dz=\frac{1}{\lambda}\int_{\gamma}\frac{\phi_{d,k}^{\varepsilon}(z)f_{d-1,k}^{-\varepsilon}(\tau)}{\left(J(\tau)-J(z)\right)} \frac{e^{\pi i z |x|^2}}{(iz\pi)^k} dz.
\]
We note that for $\tau\in SD_1,$
\[
\begin{split}
F^{\varepsilon}(x,\tau)&= \frac{1}{\lambda}\int _{w1}^{w_2}\frac{\phi_{d,k}^{\varepsilon}(z)f_{d-1,k}^{-\varepsilon}(\tau)}{\left(J(\tau)-J(z)\right)} \frac{e^{\pi i z |x|^2}}{(iz\pi)^k} dz
\\
&=\frac{1}{\lambda}\int_{\gamma}\frac{\phi_{d,k}^{\varepsilon}(z)f_{d-1,k}^{-\varepsilon}(\tau)}{\left(J(\tau)-J(z)\right)} \frac{e^{\pi i z |x|^2}}{(iz\pi)^k} dz+\varepsilon\left(\frac{i}{\tau}\right)^{2k+1}  \frac{e^{\pi i \frac{-1}{\tau} |x|^2}}{(\frac{-i}{\tau}\pi)^k}.
\end{split}
\]
We note that the right hand side is well-defined on $D_1\cup SD_1.$ Hence,
we analytically continue $F^{\varepsilon}(x,\tau)$ to $\tau \in D_1\cup SD_1$ by defining  
\begin{equation}\label{ancont}
F^{\varepsilon}(x,\tau):=\frac{1}{\lambda}\int_{\gamma}\frac{\phi_{d,k}^{\varepsilon}(z)f_{d-1,k}^{-\varepsilon}(\tau)}{\left(J(\tau)-J(z)\right)} \frac{e^{\pi i z |x|^2}}{(iz\pi)^k} dz+\varepsilon\left(\frac{i}{\tau}\right)^{2k+1}  \frac{e^{\pi i \frac{-1}{\tau} |x|^2}}{(\frac{-i}{\tau}\pi)^k}.
\end{equation}
This completes the proof of the first part of our Proposition. 
\\

Next, suppose that $\tau \in  D_1.$  By the above, we have 
\[
F^{\varepsilon}(x,\tau)=\frac{1}{\lambda}\int_{\gamma}\frac{\phi_{d,k}^{\varepsilon}(z)f_{d-1,k}^{-\varepsilon}(\tau)}{\left(J(\tau)-J(z)\right)} \frac{e^{\pi i z |x|^2}}{(iz\pi)^k} dz+\varepsilon\left(\frac{i}{\tau}\right)^{2k+1}  \frac{e^{\pi i \frac{-1}{\tau} |x|^2}}{(\frac{-i}{\tau}\pi)^k}.
\]
By equation~\eqref{eqdiff}, we have 
\[
\frac{1}{\lambda}\int_{\gamma}\frac{\phi_{d,k}^{\varepsilon}(z)f_{d-1,k}^{-\varepsilon}(\tau)}{\left(J(\tau)-J(z)\right)} \frac{e^{\pi i z |x|^2}}{(iz\pi)^k} dz -\frac{1}{\lambda}\int _{w1}^{w_2}\frac{\phi_{d,k}^{\varepsilon}(z)f_{d-1,k}^{-\varepsilon}(\tau)}{\left(J(\tau)-J(z)\right)} \frac{e^{\pi i z |x|^2}}{(iz\pi)^k} dz=\frac{e^{\pi i \tau |x|^2}}{(i\tau\pi)^k}.
\]
Therefore, we have 
\begin{equation}\label{eqnt}
F^{\varepsilon}(x,\tau)=\frac{1}{\lambda}\int _{w1}^{w_2}\frac{\phi_{d,k}^{\varepsilon}(z)f_{d-1,k}^{-\varepsilon}(\tau)}{\left(J(\tau)-J(z)\right)} \frac{e^{\pi i z |x|^2}}{(iz\pi)^k} dz+\frac{e^{\pi i \tau x^2}}{(i\pi\tau)^k} |_{2k+1}{I+\varepsilon S}
\end{equation}
for every $\tau\in D_1.$ Furthermore, we have 
\[
f_{d-1,k}^{-\varepsilon}(\tau)|_{2k+1}S=-\varepsilon  f_{d-1,k}^{-\varepsilon}(\tau),
\]
and 
\[
J(\tau)=J(\frac{-1}{\tau}).
\]
Hence,
\begin{equation}\label{eqnt1}
\frac{1}{\lambda}\int _{w1}^{w_2}\frac{\phi_{d,k}^{\varepsilon}(z)f_{d-1,k}^{-\varepsilon}(\tau)}{\left(J(\tau)-J(z)\right)} \frac{e^{\pi i z |x|^2}}{(iz\pi)^k} dz=-\varepsilon \left( \frac{i}{\tau} \right)^{2k+1}\frac{1}{\lambda}\int _{w1}^{w_2}\frac{\phi_{d,k}^{\varepsilon}(z)f_{d-1,k}^{-\varepsilon}(\frac{-1}{\tau})}{\left(J(\frac{-1}{\tau})-J(z)\right)} \frac{e^{\pi i z |x|^2}}{(iz\pi)^k} dz.
\end{equation}
Note that $\frac{-1}{\tau}\in SD_1\subset \Omega_0$, hence 
\[
F^{\varepsilon}\left(x,\frac{-1}{\tau}\right)=  \frac{1}{\lambda}\int _{w1}^{w_2}\frac{\phi_{d,k}^{\varepsilon}(z)f_{d-1,k}^{-\varepsilon}(\frac{-1}{\tau})}{\left(J(\frac{-1}{\tau})-J(z)\right)} \frac{e^{\pi i z |x|^2}}{(iz\pi)^k} dz.
\]
Finally, by \eqref{eqnt}, \eqref{eqnt1} and the above
\[
F_k^{\varepsilon}(\tau,x)|_{2k+1}{I+\varepsilon S}=\frac{e^{\pi i \tau x^2}}{(i\pi\tau)^k} |_{2k+1}{I+\varepsilon S}.
\]
This completes the proof of our Propostion.
\end{proof}
\subsection{Growth estimate}\label{proofm}
Next, we improve the exponent in Lemma~\ref{inibd} by a factor $\sin\left(\frac{\pi}{l}\right).$
\begin{proposition}\label{impbd}
We have 
\[
|a_{n,k}^{\varepsilon}(x)| \ll_k \delta^{-1}e^{ \frac{2\pi \sin\left(\frac{\pi}{l}\right)(1+\delta)}{\lambda}n}e^{-\pi \sin\left(\frac{\pi}{l}\right) |x|^2}\]
\end{proposition}
for any $\delta>0$
\begin{proof}
Let 
\[I_{\delta}:=\left\{\tau:\Im (\tau)=\sin\left(\frac{\pi}{l}\right)+\delta \text{ and } |\Re(\tau)|\leq \frac{\lambda}{2}\right\}.\]
By Proposition~\ref{ccontour},  $F^{\varepsilon}(x,\tau)$ has analytic continuation to $I_{\delta},$ and 
we have 
\[
a_{n,k}^{\varepsilon}(x)=\frac{1}{\lambda}\int_{I_{\delta}}F^{\varepsilon}(x,\tau) e^{-2\pi i \tau \frac{n}{\lambda}} d\tau.
\]
Note that $I_{\delta}\subset D_1\cup SD_1$ for small enough $\delta.$ By~\eqref{ancont}
\[
F^{\varepsilon}(x,\tau):=\frac{1}{\lambda}\int_{\gamma}\frac{\phi_{d,k}^{\varepsilon}(z)f_{d-1,k}^{-\varepsilon}(\tau)}{\left(J(\tau)-J(z)\right)} \frac{e^{\pi i z |x|^2}}{(iz\pi)^k} dz+\varepsilon\left(\frac{i}{\tau}\right)^{2k+1}  \frac{e^{\pi i \frac{-1}{\tau} |x|^2}}{(\frac{-i}{\tau}\pi)^k}.
\]
The Proposition   follows immediately from the above integration formulas.
\end{proof}
Finally, prove Theorem~\ref{mainthmder}.
\begin{proof}[Proof of Theorem~\ref{mainthmder}]
 It is easy to check  that 
  \[
  D_1\cup SD_1=\left\{\tau\in\CC:  \Im(\tau),\Im(-\frac{1}{\tau})>\sin\left(\frac{\pi}{l}\right)\right \}.
\]
Without loss of generality, we assume that 
 \[f(x)=\int \frac{e^{\pi i \tau x^2}}{(i\pi\tau)^k}d\lambda(\tau),\]  
 where $\lambda$ is a measure with bounded  variation  and  supported on a compact subset of $ D_1\cup SD_1$. We have
 \[\mathcal{F}f(x)=\int \frac{i}{\tau}\frac{e^{\pi i \frac{-1}{\tau} x^2}}{(i\pi\tau)^k}d\lambda(\tau)=\int \left(\frac{e^{\pi i \tau x^2}}{(i\pi\tau)^k}|_{2k+1}S\right)    d\lambda(\tau).\]  
 Therefore,
 \begin{equation}\label{fep}
 f^{\varepsilon}(x)= \int \left(\frac{e^{\pi i \tau x^2}}{(i\pi\tau)^k}|_{2k+1}(I+\varepsilon S)\right)    d\lambda(\tau).
 \end{equation}
 
 Moreover,  
 \[\frac{d^k}{du^k}\frac{e^{i\pi \tau |x|^2}}{(i\tau\pi)^k}=e^{i\pi \tau |x|^2},\] 
and
\[
\frac{d^k}{du^k} \mathcal{F}\left(\frac{e^{i\pi \tau |x|^2}}{(i\tau\pi)^k}\right)=\left(\frac{i}{\tau}\right)^{2k+1}e^{i\pi \frac{-1}{\tau} |x|^2},
\]
where $u=|x|^2$.  We average the above identities with respect to $d\lambda$, and obtain 
 \[
 \begin{split}
 \frac{d^k}{du^k} f(x)= \int e^{\pi i \tau |x|^2}d\lambda(\tau),
 \\
  \frac{d^k}{du^k} \mathcal{F}f(x)= \int \left( e^{\pi i \tau |x|^2}|_{2k+1}S\right)   d\lambda(\tau).
 \end{split}
 \]
 Hence,
 \begin{equation}\label{feps}
 \begin{split}
  \frac{d^k}{du^k} f^{\varepsilon}(x)=  \int \left( e^{\pi i \tau |x|^2}|_{2k+1}(I+\varepsilon S)\right)   d\lambda(\tau).
  \end{split}
 \end{equation}

 By Proposition~\ref{ccontour}, we have 
 \[
F_k^{\varepsilon}(\tau,x)|_{2k+1}{I+\varepsilon S}=\frac{e^{\pi i \tau x^2}}{(i\pi\tau)^k} |_{2k+1}{(I+\varepsilon S)}
 \]
 for every $\tau\in D_1\cup SD_1.$ By Proposition~\ref{impbd}, the Fourier expansion of $ F_k^{\varepsilon}(\tau,x)$ is convergent on  $D_1\cup SD_1,$ and we have 
 \[
F_k^{\varepsilon}(\tau,x)|_{2k+1}{(I+\varepsilon S)}=\sum_{n\geq 0}a_{n,k}^{\varepsilon}(x) e^{2\pi i \tau \frac{n}{\lambda}}|_{2k+1}{(I+\varepsilon S)}.
 \]
 This implies that  
 \[
 \sum_{n\geq 0}a_{n,k}^{\varepsilon}(x) e^{2\pi i \tau \frac{n}{\lambda}}|_{2k+1}{(I+\varepsilon S)}=\frac{e^{\pi i \tau x^2}}{(i\pi\tau)^k} |_{2k+1}{(I+\varepsilon S)}.
 \]
 We average the above identity with respect to $d\lambda$ and obtain
 \[
  \sum_{n\geq 0}a_{n,k}^{\varepsilon}(x) \int e^{2\pi i \tau \frac{n}{\lambda}}|_{2k+1}{(I+\varepsilon S)}   d\lambda= \int \frac{e^{\pi i \tau x^2}}{(i\pi\tau)^k} |_{2k+1}{(I+\varepsilon S)} d\lambda(\tau).
 \]
 We substitute the left hand side using \eqref{feps} with the values the $k$-derivatives and the right hand side using \eqref{fep} by $f^{\varepsilon}(x)$, and  obtain 
\[\sum_{n\geq d(\varepsilon,k)}a^{\varepsilon}_{n,k}(x) \frac{d^k}{du^k}  f^{\varepsilon}\left(\sqrt{\frac{2n}{\lambda}}\right)=f^{\varepsilon}(x).
\] 
This completes the proof of our theorem.
\end{proof}

\section{Domain of   $F_k^{\varepsilon}(\tau,x)$}\label{monodromy}
Let $\tilde{\mathbb{H}}$ be the universal cover of $\mathbb{H}-\Gamma w_1,$  the upper half-plan minus the orbit of $w_1$ under $\Gamma.$ It follows from the change of the contour integral that we introduce in proof of Proposition~\ref{ccontour} that  $F_k^{\varepsilon}(\tau,x)$ has an analytic continuation to the whole $\tilde{\mathbb{H}}.$ In fact, $F_k^{\varepsilon}(\tau,x)$ has non-trivial monodromy around $w_1$ and all its orbits $\Gamma w_1,$ and cannot be extended beyond $\Im(\tau)>\sin\left(\frac{\pi}{l}\right).$

\subsection{Monodromy around $w_2$}
 First, we analytically continue $F_k^{\varepsilon}(\tau,x)$ to every point in a neighborhood of $w_2$
 except the segment form $\frac{\lambda}{2}$ to $w_2.$ Let $\mathcal{D}$ be the fundamental domain for $\Gamma$; see Figure~\ref{heckef}.  Let $V=TS$ which is the hyperbolic rotation with center $w_2$ and angle $\frac{2\pi}{l}.$ We define 
 \[
\mathcal{D}_i=V^{i}\left(\mathcal{D}\cup T\mathcal{D}-\{w_1,w_2\}   \right).
\]
Note that $\mathcal{D}_{i+p}=\mathcal{D}_i.$ We note by Proposition~\ref{ccontour},   $F_k^{\varepsilon}(\tau,x)$ is analytic and well defined on $\mathcal{D}_0=\mathcal{D}\cup T\mathcal{D}-\{w_1,w_2\} \subset \Omega_0.$ We denote this restriction by $F(x,\tau).$ By changing the contour integral, $F$ has an analytic continuation to $S\mathcal{D} \subset \mathcal{D}_{p-1}$ which satisfies 
\[
F(\tau,x)+\varepsilon \left(\frac{i}{\tau}\right)^{2k+1}F\left(\frac{-1}{\tau},x\right)=\frac{e^{\pi i \tau |x|^2}}{(i\tau\pi)^k}|_{2k+1}^{\varepsilon}{I+ S},
\]
\[
F(\tau,x)=F(\tau+\lambda,x)
\]
for every $\tau\in \mathcal{D}.$ By combining the above identities, we have 
\[
F(\alpha,x)|_{2k+1}^{-\varepsilon}V^{-1}=F(\alpha,x)+\frac{e^{\pi i \alpha |x|^2}}{(i\alpha\pi)^k}|_{2k+1}^{-\varepsilon}V^{-1}-T^{-1}.
\]

for every $\alpha=\tau+\lambda \in T\mathcal{D}.$ It is clear from the above functional equation that $F$ has an analytic continuation to $\mathcal{D}_{p-1}=V^{-1}(\mathcal{D}_0)$ by choosing $\alpha\in \mathcal{D}_0.$ 
\\

By applying the slash operator to the above,  we have 
\[
F(\alpha,x)|_{2k+1}^{-\varepsilon}V^{-(i+1)}=F(\alpha,x)|_{2k+1}^{-\varepsilon}V^{-i}+\frac{e^{\pi i \alpha |x|^2}}{(i\alpha\pi)^k}|_{2k+1}^{-\varepsilon}V^{-(i+1)}-T^{-1}V^{-i}.
\]
By the above functional equation, we analytically continue $F$ to $\mathcal{D}_{l-2}$ and $\mathcal{D}_{p-3}, \dots$ recursively. Hence, by induction 
\begin{equation}\label{-Vext}
F(\alpha,x)|_{2k+1}^{-\varepsilon}V^{-n}=F(\alpha,x)+\frac{e^{\pi i \alpha |x|^2}}{(i\alpha\pi)^k}|_{2k+1}^{-\varepsilon}\sum_{i=0}^{n-1}V^{-(i+1)}-T^{-1}V^{-i}.
\end{equation}
Similarly, we extend $F$ to $\mathcal{D}_1$ by writing
\[
F(\tau,x)+\varepsilon \left(\frac{i}{\tau}\right)^{2k+1}F\left(\frac{-1}{\tau},x\right)=\frac{e^{\pi i \tau |x|^2}}{(i\tau\pi)^k}|_{2k+1}^{\varepsilon}{I+ S},
\] 
\[
F\left(\frac{-1}{\tau},x\right)= F\left(\frac{-1}{\tau}+\lambda,x\right)
\]
for $\tau\in \mathcal{D}.$ By combining the above, we have 
\[
F(\alpha,x)|_{2k+1}^{-\varepsilon}V=F(\alpha,x)+\frac{e^{\pi i \alpha |x|^2}}{(i\alpha\pi)^k}|_{2k+1}^{-\varepsilon}T^{-1}V-I.
\]
By applying the slash operator to the above and induction,  we have 
\begin{equation}\label{Vext}
F(\alpha,x)|_{2k+1}^{-\varepsilon}V^n=F(\alpha,x)+\frac{e^{\pi i \alpha |x|^2}}{(i\alpha\pi)^k}|_{2k+1}^{-\varepsilon}\sum_{i=0}^{n-1}T^{-1}V^{i+1}-V^{i}.
\end{equation}
\eqref{-Vext} and~\eqref{Vext}, give two different extension of $F$ on $\mathcal{D}_3$  for $n=\pm3.$ We obtain 
\[
F(\alpha,x)|_{2k+1}^{-\varepsilon}V^3-F(\alpha,x)|_{2k+1}^{-\varepsilon}V^{-3}=\frac{e^{\pi i \alpha |x|^2}}{(i\alpha\pi)^k}|_{2k+1}^{-\varepsilon}\sum_{i=0}^{5}T^{-1}V^{i+1}-V^{i}.
\]
for $\alpha \in \mathcal{D}_0.$ Recall that 
\[
r^{\varepsilon}_k(\gamma,\tau;x):=\frac{e^{i\pi \tau |x|^2}}{(i\pi\tau)^k}|^{-\varepsilon}_{2k+1} (T^{-1}-I)(1+V+\dots+V^{l-1})\gamma.
\]

By the above, we may consider
$r^{\varepsilon}_k(\gamma,\tau;x)$ as the obstruction for the analytic continuation of $F_k^{\varepsilon}(\tau,x)$.  
\\

Finally we give a proof of of Theorem~\ref{mainthm2} after two an auxiliary lemmas. Recall that   $u=|x|^2.$ 
\begin{lemma}\label{difslsh}
Let $\frac{p(\tau)}{q(\tau)}$ be a rational function. 
We have 
\[
\frac{d^l}{du^l} \left(\left(\frac{p(\tau)}{q(\tau)}e^{i\pi \tau |x|^2} \right)|^{-\varepsilon}_{2k+1} \gamma \right)= \left(\frac{d^l}{du^l} \left(\frac{p(\tau)}{q(\tau)}e^{i\pi \tau |x|^2} \right) \right)|^{-\varepsilon}_{2k+1} \gamma
\]
for every $\gamma\in PSL_2(\R)$ and $l\geq 0.$
\end{lemma}
\begin{proof}
We write 
\(
\gamma(\tau)
\) for the action of $\gamma$ on $\tau\in\HH.$
We have 
\[
\left(\frac{p(\tau)}{q(\tau)}e^{i\pi \tau |x|^2} \right)|^{-\varepsilon}_{2k+1} \gamma=j^{-\varepsilon}_{2k+1}(\tau,\gamma) \frac{p(\gamma(\tau))}{q(\gamma(\tau))}e^{i\pi \gamma(\tau) |x|^2}.
\]
Then
\[
\begin{split}
\frac{d^l}{du^l} \left(\left(\frac{p(\tau)}{q(\tau)}e^{i\pi \tau |x|^2} \right)|^{-\varepsilon}_{2k+1} \gamma \right)&=j^{-\varepsilon}_{2k+1}(\tau,\gamma) \frac{p(\gamma(\tau))}{q(\gamma(\tau))} (i\pi \gamma(\tau))^l e^{i\pi \gamma(\tau) |x|^2}
\\
&=\left((i\pi \tau)^l \frac{p(\tau)}{q(\tau)}e^{i\pi \tau |x|^2} \right)|^{-\varepsilon}_{2k+1} \gamma
\\
&=\left(\frac{d^l}{du^l} \left(\frac{p(\tau)}{q(\tau)}e^{i\pi \tau |x|^2} \right) \right)|^{-\varepsilon}_{2k+1} \gamma.
\end{split}
\]

\end{proof}

\begin{lemma}\label{fourr}
Let $\mathcal{F}(f(x))$ be fourier transformation with respect to $x\in\R^2.$
We have 
\[
\mathcal{F}\left(\frac{e^{i\pi \tau |x|^2}}{(i\pi\tau)^k}|^{-\varepsilon}_{2k+1} \gamma\right) (\xi)  =-\varepsilon \frac{e^{i\pi \tau |\xi|^2}}{(i\pi\tau)^k} |^{-\varepsilon}_{2k+1} S\gamma
\]
\end{lemma}
\begin{proof}
We have 
\[
\begin{split}
\mathcal{F}\left(\frac{e^{i\pi \tau |x|^2}}{(i\pi\tau)^k}|^{-\varepsilon}_{2k+1} \gamma \right) (\xi)&=j^{-\varepsilon}_{2k+1}(\tau,\gamma) \frac{1}{(i\pi\gamma(\tau))^k}\mathcal{F} \left(e^{i\pi \gamma(\tau) |x|^2}\right) (\xi).
\\
&=j^{-\varepsilon}_{2k+1}(\tau,\gamma)\frac{1}{(i\pi\gamma(\tau))^k}\frac{i}{\gamma(\tau)} e^{i\pi \frac{-1}{\gamma(\tau)} |\xi|^2}.
\end{split}
\]
We have 
\[
\begin{split}
\frac{e^{i\pi \tau |\xi|^2}}{(i\pi\tau)^k} |^{-\varepsilon}_{2k+1} S\gamma&=-\varepsilon \left(\frac{i}{\tau} \right)^{2k+1}\frac{e^{i\pi \frac{-1}{\tau} |\xi|^2}}{(i\pi\frac{-1}{\tau})^k} |^{-\varepsilon}_{2k+1} \gamma
\\
&= -\varepsilon   j^{-\varepsilon}_{2k+1}(\tau,\gamma) \frac{1}{(i\pi\gamma(\tau))^k} \frac{i}{\gamma(\tau)} e^{i\pi \frac{-1}{\gamma(\tau)} |\xi|^2}.
\end{split}
\]
The lemma follows from  the above identities.
\end{proof}

\begin{proof}[Proof of Theorem~\ref{mainthm2} ]
First we check that 
\[
\frac{d^k}{du^k}r_k^{\varepsilon}\left(\gamma,\tau;\sqrt{\frac{2n}{\lambda}}\right)=0
\]
for $|x|=\sqrt{\frac{2n}{\lambda}}.$ By Lemma~\ref{difslsh}, we have 
\[
\frac{d^k}{du^k}r^{\varepsilon}_k(\gamma,\tau;x)=\left( \left(\frac{d^k}{du^k}\frac{e^{i\pi \tau |x|^2}}{(i\pi\tau)^k}\right)|^{-\varepsilon}_{2k+1} (T^{-1}-I)  \right)|^{-\varepsilon}_{2k+1}  (1+V+\dots+V^{l-1})\gamma.
\]
For the inside function of the right hand side, we have 
\[
\left(\frac{d^k}{du^k}\frac{e^{i\pi \tau |x|^2}}{(i\pi\tau)^k}\right)|^{-\varepsilon}_{2k+1} (T^{-1}-I) =e^{i\pi \tau |x|^2}|^{-\varepsilon}_{2k+1} (T^{-1}-I).
\]
We have 
\[
e^{i\pi \tau |x|^2}|^{-\varepsilon}_{2k+1} (T^{-1}-I)= e^{i\pi \tau \frac{2m}{\sqrt{3}}}|^{-\varepsilon}_{2k+1} (T^{-1}-I)=0.
\]
for  $|x|=\sqrt{\frac{2n}{\lambda}}.$ This completes the proof of the first part of Theorem~\ref{mainthm2}. 
\\

Next we prove the other part. By Lemma~\ref{fourr}
\[
\mathcal{F}r^{\varepsilon}_k(\gamma,\tau;x)=-\varepsilon \frac{e^{i\pi \tau |x|^2}}{(i\pi\tau)^k}|^{-\varepsilon}_{2k+1} S(T^{-1}-I)(1+V+\dots+V^{l-1})\gamma.
\]
We note that 
\[
S(T^{-1}-I)(1+V+\dots+V^{l-1})\gamma=\sum_{i=0}^{l-1}V^{i}-T^{-1}V^{i+1}=-(T^{-1}-I)(1+V+\dots+V^{l-1})\gamma.
\]
Therefore,
\[
\mathcal{F}r^{\varepsilon}_k(\gamma,\tau;x)=\varepsilon \frac{e^{i\pi \tau |x|^2}}{(i\pi\tau)^k}|^{-\varepsilon}_{2k+1} S(T^{-1}-I)(1+V+\dots+V^{l-1})\gamma=\varepsilon r^{\varepsilon}_k(\gamma,\tau;x).
\]
This completes the proof of Theorem~\ref{mainthm2}.
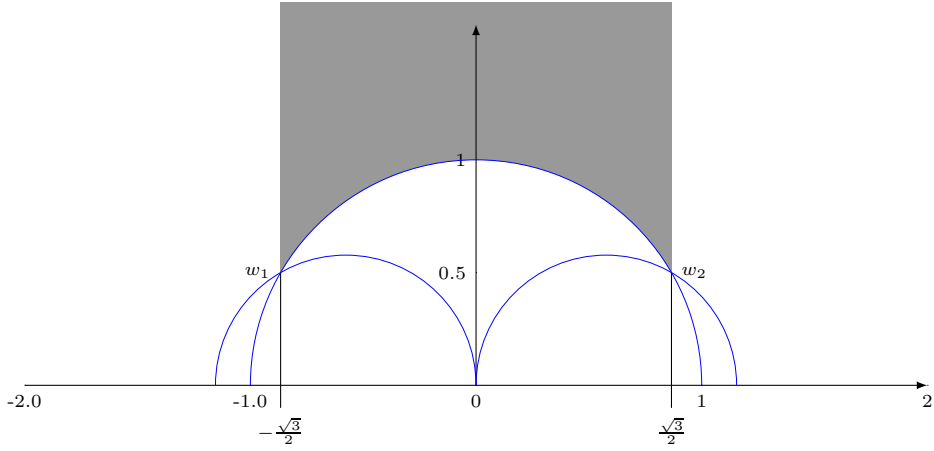
\begin{figure}
\centering
\begin{tikzpicture}[scale=3]
    \draw[-latex](\myxlow,0) -- (\myxhigh ,0);
    \pgfmathsetmacro{\succofmyxlow}{\myxlow+1}
    \foreach \x in {\myxlow,\succofmyxlow,...,\myxhigh}
    {   \draw (\x,-0) -- (\x,-0.0) node[below,font=\tiny] {\x};
    }
        {   \draw (-0.866,-0.1)node[below,font=\tiny]{$-\frac{\sqrt{3}}{2}$}--(-0.866,0.5)node[left,font=\tiny]{$w_1$}  (0.866,0.5)node[right,font=\tiny]{$w_2$} -- (0.866,-0.1) node[below,font=\tiny] {$\frac{\sqrt{3}}{2}$};
    }
    
    \foreach \y  in {0.5,1}
    {   \draw (0,\y) -- (-0.0,\y) node[left,font=\tiny] {\pgfmathprintnumber{\y}};
    }
    \draw[-latex](0,-0.0) -- (0,1.6);
    \begin{scope}   
        \clip (\myxlow,0) rectangle (\myxhigh,1.2);
            {   \draw[very thin, blue] (1,0) arc(0:180:1);
            }   
                        {   \draw[very thin, blue] (0,0) arc(0:180:0.57735);
            }   
                                    {   \draw[very thin, blue] (1.1547,0) arc(0:180:0.57735);
            }   
    \end{scope}
    \begin{scope}
            \begin{pgfonlayer}{background}
            \clip (-0.866,0) rectangle (0.866,1.7);
            \clip   (1,1.7) -| (-1,0) arc (180:0:1) -- cycle;
            \fill[gray,opacity=0.8] (-1,-1) rectangle (1,2);
        \end{pgfonlayer}
    \end{scope}
\end{tikzpicture}
\captionof{figure}{Fundamental domain for the Hecke triangle $(2,6,\infty)$}\label{heckeff}
\end{figure}

\end{proof}
\section{Proof of the Conjecture of Cohn et al.}\label{zeros}
\subsection{Proof of Theorem~\ref{mainconj}}
 In this section, we show that Theorem~\ref{strongthm} implies Theorem~\ref{mainconj}. 

\begin{lemma}\label{Acons}
Let $\delta>0$ be any positive real number.  There exists a periodic subset of integers $\tilde{A}\subset\Z$ such that if $n=x^2+xy+y^2$ then $n\in \tilde{A}$ and the density of $\tilde{A}$ is smaller than $\delta.$
\end{lemma}
\begin{proof}
Suppose that 
\(
n=x^2+xy+y^2
\)
and  $ord_p(n)=2k+1$, where $l$ is a prime number.
It is an elementary fact that  $l\equiv 1 \mod(3)$. Let 
\[
L:=\prod_{\substack{p<M\\ l\equiv -1 \mod 3}} p^2.
\]
Let $B\subset \Z/L\Z$ be the subset of congruence class mod $L$  that are congruent to some $n=x^2+xy+y^2$ mod $L.$ It follows that
\[
\frac{|B|}{L}=\prod_{\substack{p<M\\ l\equiv -1 \mod 3}}\frac{p^2-p+1}{p^2}=\prod_{\substack{p<M\\ l\equiv -1 \mod 3}}\left(1-\frac{1}{l}+\frac{1}{p^2}   \right)=O(\log(M)^{-1/2}).
\]
This implies that the density of $B$ could be as small as possible by taking $M$ large enough. Let 
\[
\tilde{A}=\{a\in \Z^+: \bar{a}\in B, \text{ where } a\equiv \bar{a} \mod L \}.
\]
This completes the proof of our lemma. 
\end{proof}
Next, we show that Theorem~\ref{strongthm} implies Theorem~\ref{mainconj}. 

\begin{proof}[Proof of Theorem~\ref{mainconj}] Let $L$ and $\tilde{A}$ be as  in the proof of Lemma~\ref{Acons}.
Let 
\[
A=\{n>100: n\in \tilde{A}  \}.
\]
Since the density of integers $n=x^2+xy+z^2$ is zero among all integers and $A$ has density $\frac{l}{L}.$ There exists infinity many $a_1, a_2, \dots $ such that $a_i\in A$ and $a_i\neq x^2+xy+y^2.$ By Theorem~\ref{strongthm}, there exists radial Schwartz function $f_i$ such that $f_i$ and $\mathcal{F}(f_i)$ vanish of order $2$ on all integers $n=x^2+xy+z^2,$ where $n>100.$ By linear algebra, there exists a finite linear combination 
\[
f=\sum \alpha_i f_i
\] 
such that $f$ and $\mathcal{F} f$ vanish on all integers $n=x^2+xy+z^2.$ Since $f_i$ are linearly independent $f\neq 0.$ It is clear that there are infinitely many linear independent radial Schwartz function $f$ as above. This completes the proof of Theorem~\ref{mainconj}.
\end{proof}
\subsection{Proof of Theorem~\ref{strongthm}}\label{reduc}
We briefly discuss our proof Theorem~\ref{strongthm}. Recall that
\[
A:=\left\{a>100 | a\equiv a_i \mod L, \text{ for some }a_i \text{ where } 1\leq i\leq l\right\}.
\]
where $\frac{l}{L}<0.001.$  
\[
r^{\varepsilon}(\tau;x):=e^{i\pi \tau |x|^2}|^{\varepsilon}_1 (T^{-1}-I)(1+V+\dots+V^5),
\]
and 
\[
s^{\varepsilon}(\tau;x):=e^{i\pi \tau |x|^2}|^{\varepsilon}_1 (1+V+\dots+V^5).
\]

By Corollary~\ref{mainthm22}
\[
r^{\varepsilon}\left(\tau;\sqrt{\frac{2m}{\sqrt{3}}}\right)=0.
\]
for every  integer $m\geq0.$ By Lemma~\ref{difslsh}, we have 
\[
\frac{d}{du}r^{\varepsilon}\left(\tau;\sqrt{\frac{2m}{\sqrt{3}}}\right)=c s^{\varepsilon}\left(\tau;\sqrt{\frac{2m}{\sqrt{3}}}\right),
\]
where $c=\pi i.$ Our idea is to average $r^{\varepsilon}(\tau;x)$ over $\tau$ with respect to a linear combination of probability measures supported on a compact region of the upper half-plane   such that for every $m\in A-\{a\}$ the derivatives at  $\sqrt{\frac{2m}{\sqrt{3}}}$ vanishes. More precisely, let 
\[
f(x)=\int r^{\varepsilon}(\tau;x) d\mu(\tau).
\]
Then 
\[
f\left(\sqrt{\frac{2m}{\sqrt{3}}}\right)=0, \text{ and } \frac{d}{du}f\left(\sqrt{\frac{2m}{\sqrt{3}}}\right)=c\int s^{\varepsilon}(\tau;x) d\mu(\tau).
\]
We construct $\mu$  as the weak$^*$ limit of a sequence of measures $\{\mu_n\}$ such that 
\[
\int s^{\varepsilon}\left(\tau;\sqrt{\frac{2}{\sqrt{3}}m}\right) d\mu_n(\tau)=0,
\]
and 
\[
\int s^{\varepsilon}\left(\tau;\sqrt{\frac{2}{\sqrt{3}}a}\right) d\mu_n(\tau)= 1,
\]
where $m\in A-\{a\}$ and $0<m<n.$
The existence of a  weak$^*$ limit is a consequence of the compactness of the space of probability measures  on a compact Borel measure space.

 \subsection{Construction of $\mu_n$}
In this section, we construct $\mu_n.$ First, we define a map from the space of probability measures to a finite dimensional vector space. Let $A_n\subset A$ be subset of integers  $m\in A$ where $0<m<n$. Let $H_n:=\CC^{\#A_n}$ where each coordinate is indexed by $m\in A_n$. Let
\[
z=\frac{\sqrt{3}}{2}+x+it,
\]
where $0<t<\frac{1}{2}.$ Then 
\[
\Im(V^{\pm}z_t)=\frac{t}{(\frac{\sqrt{3}}{2}\pm x)^2+t^2}. 
\]
We have 
\[
\Im(V^{\pm}z_t)-\Im(z_t)=\frac{t}{(\frac{\sqrt{3}}{2}\pm x)^2+t^2}-t. 
\]
We note that the maximum of the above function on the interval $0<t<1/2$ and small $x\sim \delta$ is at $t_0\sim 0.27$  with value $0.058.$ Let 
\[
\X_{\delta}:= \left\{\sqrt{3}(\frac{1}{2}+x)+it_0: |x|<\delta\right\}.
\]
\begin{lemma}\label{apprlem}
Let $t_0=0.27$, $\delta<0.01$ and $\tau\in \X_{\delta}$  then for every $1\leq i\leq 5,$ we have 
\[
\Im(V^{i}\tau)-\Im(\tau)\geq 0.05. 
\]
\end{lemma}
\begin{proof}
It is easy to check it numerically. 
\end{proof}
The construction of $\mu_n$ involves both a discrete and continuous averaging that we discuss next.

\subsubsection{Continuous averaging}
We identify the unit circle $S^1$ with $2\pi \theta$ where $0\leq \theta<1.$ 
\begin{proposition}\label{discprop}
For $n\in \Z^+$ there exits a probability measure $\lambda_{n,\alpha}(\theta)d\theta$ on $S^1$ such that 
\[
\int_{0}^1 e^{2\pi i m\theta} \lambda_{n,\alpha}(\theta)d\theta=\begin{cases} 1 &\text{ if } m=0,\\  \frac{1}{2} e^{ \pm i\alpha} &\text{ if } m=\pm n, \\ 0 &\text{ if } m\neq 0,\pm n.  \end{cases}
\]
where $m\in \Z.$
\end{proposition}
\begin{proof} 
Let 
\[
\lambda_{n,\alpha}( \theta)d \theta=(1+\cos(2\pi n \theta-\alpha ))d\theta.
\]
It is clear that $\lambda_{n,\alpha}(\theta)\geq 0$ and $\int_{S^1} \lambda_{n,\alpha}( \theta)d \theta=1.$ The moment identities follows from orthogonality of $e^{2\pi i m\theta}$ on the unit circle. 
\end{proof}

\subsubsection{Discrete averaging} Recall that 
\[
A:=\{a>100| a\equiv a_k \mod L, \text{ for some }a_k \text{ where } 1\leq k\leq l\}.
\]
Let
\[
\A:=\begin{bmatrix} a_{j,k}  \end{bmatrix}_{1\leq j,k\leq l} \in [\CC]_{l\times l}.
\]
where $a_{j,k}:=e^{2\pi i \frac{j}{L}a_k}.$
It follows from the Vandermonde determinant that  $\A$ is invertible. Let 
\[\A^{-1}:=\begin{bmatrix}\alpha_{k,j} \end{bmatrix}_{1\leq k,j\leq l}.\]
Then 
\[
\sum_{j}\alpha_{k,j} e^{2\pi i a_s\frac{j}{L}} = \delta_{k,s}.
\]
\subsubsection{Main result}
Let $\PP$ be the space of probability measure on the disjoint union of $l$ unit intervals $[0,1]$. We parametrize elements of $\mu\in \PP$ by $\mu:=(\Lambda(1)\lambda_1,\dots,\Lambda(l)\lambda_l)$ where $\lambda_i$ are probability measures on the unit interval $[0,1]$ and $\Lambda(i)\geq 0$ are non-negative real numbers where $\sum \Lambda(i)=1.$ For an integer $ 1\leq s \leq l,$ let 
   $I(s)\subset X_{\frac{l}{L}} $ be:
  \[
  I(s):=\left\{\sqrt{3}\left(\frac{1}{2}+\frac{x}{L}+\frac{s-1}{L}\right)+it_0: 0\leq x\leq 1 \right\}.
  \]
 Given $\mu:=(\Lambda(1)\lambda_1,\dots,\Lambda(l)\lambda_l)\in \PP,$ we define the following  measure on the upper half-plane supported on $ I(s)$
\[
d\eta(s,\mu,\tau) := L\sum_{k}  \Lambda(k) \alpha_{k,s} e^{-2 \pi i a_k(\frac{1}{2} +\frac{x}{L})}\lambda_k(x)dx.
\]
We define 
\[
d\eta(\mu,\tau)  =\sum_{s=1}^l d\eta(s,\mu,\tau),
\]
and  $\psi_n:\PP \to H_n$ as follows
\[
\psi_n(\mu):=\begin{bmatrix} \int s^{\varepsilon}\left(\tau;\sqrt{\frac{2}{\sqrt{3}}m}\right) d\eta(\mu,\tau)    \end{bmatrix}_{m\in A_n}.
\]

  Let $(\lambda\times \Lambda)\in \PP$ where 
$\lambda(\theta)d\theta$ is a continuous probability measure on the unit circle, and $\Lambda$ a probability measure on the finite set $\{k:1\leq k\leq l\}.$  
Let $\lambda_{n,\alpha}$ be the probability measure constructed in Proposition~\ref{discprop} and $\Lambda_k$ be the discrete measure with mass $1$ at $k$ and $0$ at other points.
\begin{proposition}\label{mainest}
We have 
\[
\int s^{\varepsilon}\left(\tau;\sqrt{\frac{2}{\sqrt{3}}(Lm+a_j)}\right) d\eta(\lambda_{n,\alpha}\times\Lambda_k,\tau) = \frac{e^{i \alpha}}{2} e^{-\pi t_0 \frac{2}{\sqrt{3}}(Lm+a_j)}\left(\delta_{j,k}\delta_{n,m}+O(e^{-0.05 \pi \frac{2}{\sqrt{3}}(Lm+a_j)})   \right).
\]
\end{proposition}
\begin{proof}
Suppose that $\tau \in \X_{\delta}.$ By Lemma~\ref{apprlem}, we have  
\[\Im(V^{i}\tau)-\Im(\tau)\geq 0.05. \]
We have
\[
\begin{split}
s^{\varepsilon}\left(\tau;\sqrt{\frac{2}{\sqrt{3}}(Lm+a_j)}\right)&=e^{i\pi \tau \frac{2}{\sqrt{3}}(Lm+a_j)}|^{\varepsilon}_1 (1+V+\dots+V^5)
\\
&=e^{i\pi \tau \frac{2}{\sqrt{3}}(Lm+a_j)} \left(1+ O(e^{-0.05 \pi \frac{2}{\sqrt{3}}(Lm+a_j)})   \right).
\end{split}
\]
By the above estimate, we have 
\[
\begin{split}
\int s^{\varepsilon}\left(\tau;\sqrt{\frac{2}{\sqrt{3}}(Lm+a_j)}\right)d\eta(\lambda_{n,\alpha}\times\Lambda_k,\tau) &=\int e^{i\pi \tau \frac{2}{\sqrt{3}}(Lm+a_j)}d\eta(\lambda_{n,\alpha}\times\Lambda_k,\tau) 
\\&+O\left(e^{-(t_0+0.05) \pi \frac{2}{\sqrt{3}}(Lm+a_j)}\right).
\\
&
\end{split}
\]
Next, we estimate the main term of the right hand side of the above identity. We have 
\[
\begin{split}
\int e^{i\pi \tau \frac{2}{\sqrt{3}}(Lm+a_j)}&d\eta(\lambda_{n,\alpha}\times\Lambda_k,\tau) 
\\
&=Le^{-t_0 \pi \frac{2}{\sqrt{3}}(Lm+a_j)} \sum_{s=0}^{l-1}\int_{0}^{\frac{1}{L}} e^{2\pi i (\frac{1}{2}+\frac{s}{L}+x) (Lm+a_j)} \alpha_{k,s+1} e^{-2 \pi i  a_k (\frac{1}{2}+x)}\lambda_{n,\alpha}(Lx)dx
\\
&=e^{-t_0 \pi \frac{2}{\sqrt{3}}(Lm+a_j)} \sum_{s=0}^{l-1}  e^{2\pi i \frac{sa_j}{L}}\alpha_{k,s+1} 
\\&L\int_{0}^{\frac{1}{L}} e^{2\pi i (\frac{1}{2}+x) (Lm+a_j)} e^{-2 \pi i  a_k (\frac{1}{2}+x)}\lambda_{n,\alpha}(Lx)dx
\\
&=e^{-t_0 \pi \frac{2}{\sqrt{3}}(Lm+a_j)} \delta_{j,k}\int_{0}^{1} e^{2\pi i m\theta} \lambda_{n,\alpha}(\theta)d\theta= \frac{e^{i \alpha}}{2}  e^{-t_0 \pi \frac{2}{\sqrt{3}}(Lm+a_j)} \delta_{j,k}\delta_{m,n}.
\end{split}
\]
where $\theta=Lx.$ This proves  our proposition.
\end{proof}

We state the main Proposition  of this section. Suppose that $a\leq n$ and define  $\vec{a}_n \in H_n $
\[
\vec{a}_n:=e^{-\pi (t_0+0.05 )\frac{2}{\sqrt{3}}a} [\delta_{m,a}]_{m\in A_n}.
\]
\begin{proposition}\label{mainpropp}
There exists $\mu_n\in \PP$ such that $\psi_n(\mu_n)=\vec{a}_n.$
\end{proposition}
\begin{proof}
We consider $H_n$ as a real vector space of dimension $2\#A_n$ (each complex number has two real coordinates $z=(z_1,z_2)$ where $z_1=\text{Re}(z)$ and $z_2=\Im(z)).$ 
We note that $\psi_n(\PP)\subset H_n$ is a convex subset. Suppose the contrary that $\vec{a}_n\notin \psi_n(\PP).$ Then 
 there there exits a hyperplane  that separate $\psi_n(\PP)$ from $\vec{a}_n$. In other words there exits a unit vector $u\in H_n$ such that 
 \[
 \left< u,\psi(\mu)- \vec{a}_n \right> >0
 \]
 for every $\mu\in \mathcal{P}.$
 Suppose that $u=[(u_{(l,1)},u_{(l,2)}]_{l\in A_n},$ and 
\begin{equation} \label{maxass}e^{-\pi t_0 \frac{2}{\sqrt{3}}b} \sqrt{u_{(b,1)}^2 + u_{(b,2)}^2}=\max\left(e^{-\pi t_0 \frac{2}{\sqrt{3}}l} \sqrt{u_{(l,1)}^2 + u_{(l,2)}^2}\right)_{l\in A_n}.\end{equation}

Suppose that 
\[
b=nL+a_k.
\]
Let 
\begin{equation}\label{defalph}
e^{i \alpha}=-\frac{u_{(b,1)}+iu_{(b,2)}}{\sqrt{u_{(b,1)}^2+u_{(b,2)}^2}},
\end{equation}
 and 
\[
\mu_n:=\lambda_{n,\alpha}\times\Lambda_k.
\]
By Proposition~\ref{mainest}, for $l\neq b $ we have 
\begin{equation}\label{inq1}
|(\psi(\mu_n)-\vec{a}_n) _{(l,j)}| \leq  e^{-\pi (t_0+0.05) \frac{2}{\sqrt{3}}l}.
\end{equation}
By \eqref{defalph} and Proposition~\ref{mainest},  we have 
\begin{equation}\label{inq2}
(\psi(\mu_n)-\vec{a}_n)_{(b,1)}u_{(b,1)}+(\psi(\mu_n)-\vec{a}_n)_{(b,2)}u_{(b,2)} =-e^{-\pi t_0 \frac{2}{\sqrt{3}}b} \sqrt{u_{(b,1)}^2 + u_{(b,2)}^2}+O\left(e^{-\pi (t_0+0.05) \frac{2}{\sqrt{3}}b}\right).
\end{equation}
We have
\[
\begin{split}
 \left< u,\psi(\mu_n)-\vec{a}_n \right>&= (\psi(\mu_n)-\vec{a}_n)_{(b,1)}u_{(b,1)}+(\psi(\mu_n)-\vec{a}_n)_{(b,2)}u_{(b,2)}
 \\
 &+\sum_{j,l\neq b} (\psi(\mu_n)-\vec{a}_n)_{(l,j)} u_{(l,j)}.
 \end{split}
 \]
 By inequalities~\eqref{inq1} and \eqref{maxass}
\[
\begin{split}
 \left|\sum_{j,l\neq b} (\psi(\mu_n)-\vec{a}_n)_{(l,j)} u_{(l,j)}\right| &\leq \sum_{j,l\neq b}e^{-\pi (t_0+0.05) \frac{2}{\sqrt{3}}l} |u_{(l,j)}|
 \\
 &\leq  e^{-\pi t_0 \frac{2}{\sqrt{3}}b} \sqrt{u_{(b,1)}^2 + u_{(b,2)}^2} \sum_{j,l\neq b}e^{-\pi (0.05) \frac{2}{\sqrt{3}}l}.
 \end{split}
  \]
  We note that $l>100$, hence
  \[
   \sum_{j,l\neq b}e^{-\pi (0.05) \frac{2}{\sqrt{3}}l} \leq  e^{-\pi (0.05) \frac{2}{\sqrt{3}}100} 5.52.
  \]

Note that the probability measure $\mu_n\in \PP$ is such that 
\[
\sum_{j,l\neq b}| (\psi(\mu_n)-\vec{a}_n)_{(l,j)} u_{(l,j)}|<  | (\psi(\mu_n)-\vec{a}_n)_{(b,1)}u_{(b,1)}+ (\psi(\mu_n)-\vec{a}_n)_{(b,2)}u_{(b,2)}|,
\]
and 
\[
 (\psi(\mu_n)-\vec{a}_n)_{(b,1)}u_{(b,1)}+ (\psi(\mu_n)-\vec{a}_n)_{(b,2)}u_{(b,2)}<0.
\]
This implies that 
\[
 \left< u, \psi(\mu_n)-\vec{a}_n \right> <0,
\]
which is a contradiction.

\end{proof}
Finally, we give a proof of Theorem~\ref{strongthm}.
\begin{proof}[Proof of Theorem~\ref{strongthm}]
Recall that $\PP$ is the space of probability measure on the disjoint union of $l$ unit intervals $[0,1]$. By Proposition~\ref{mainpropp}, there exists a sequence of probability measures 
$\mu_n\in \PP$ for every $n>a$ such that 
\[
\int s^{\varepsilon}\left(\tau;\sqrt{\frac{2}{\sqrt{3}}m}\right) d\eta(\mu_n,\tau) =e^{-\pi (t_0+0.05 )\frac{2}{\sqrt{3}}a} \delta_{m,a},
\]
where $m\in A_n.$
Since $\PP$ is compact with respect to the weak$^*$ topology. There exists $\mu$ which is a weak$^*$ limit of a subsequence of $\mu_{n}.$ It is clear that 
\[
\int s^{\varepsilon}\left(\tau;\sqrt{\frac{2}{\sqrt{3}}m}\right) d\eta(\mu,\tau) =e^{-\pi (t_0+0.05 )\frac{2}{\sqrt{3}}a} \delta_{m,a}
\]
for every $m\in A.$ We define 
\[
f(x):=e^{\pi (t_0+0.05 )\frac{2}{\sqrt{3}}a} \int  r^{\varepsilon}(\tau;x)d\eta(\mu,\tau).
\]
It follows from the reduction that we discuss in section~\ref{reduc} that $f$ satisfies that properties of Theorem~\ref{strongthm}.
\end{proof}

\bibliographystyle{alpha}
\bibliography{inter}

\end{document}